\documentclass[a4paper]{amsart}
\usepackage{amsmath,amsthm,amssymb,latexsym,epic,bbm,comment,mathbbol}
\usepackage{graphicx,enumerate,stmaryrd,color}
\usepackage[all,2cell]{xy}
\xyoption{2cell}

\newtheorem{theorem}{Theorem}
\newtheorem{lemma}[theorem]{Lemma}
\newtheorem{corollary}[theorem]{Corollary}
\newtheorem{proposition}[theorem]{Proposition}

\newtheorem{remark}[theorem]{Remark}
\usepackage[all]{xy}
\usepackage[active]{srcltx}
\usepackage[parfill]{parskip}
\usepackage{enumerate}

\font\sc=rsfs10
\newcommand{\cC}{\sc\mbox{C}\hspace{1.0pt}}

\newcommand{\cI}{\sc\mbox{I}\hspace{1.0pt}}

\newcommand{\cS}{\sc\mbox{S}\hspace{1.0pt}}

\font\scc=rsfs7
\newcommand{\ccC}{\scc\mbox{C}\hspace{1.0pt}}

\newcommand{\ccS}{\scc\mbox{S}\hspace{1.0pt}}

\begin{document}

\title[Simple transitive 2-representations for Soergel bimodules]
{Simple transitive $2$-representations\\ of small quotients of Soergel bimodules}
\author[T.~Kildetoft, M.~Mackaay, V.~Mazorchuk and J.~Zimmermann]{Tobias 
Kildetoft, Marco Mackaay,\\  Volodymyr Mazorchuk and Jakob Zimmermann}

\begin{abstract}
In all finite Coxeter types but $I_2(12)$, $I_2(18)$ and $I_2(30)$, 
we classify simple transitive $2$-rep\-re\-sen\-ta\-ti\-ons 
for the quotient of the $2$-category of Soergel bimodules over the coinvariant 
algebra which is associated to the two-sided cell that is the closest one to the 
two-sided cell  containing the identity element.  It turns out that, in most of 
the cases, simple transitive $2$-representations are exhausted by cell $2$-representations.
However, in Coxeter types $I_2(2k)$, where $k\geq 3$, there exist simple transitive 
$2$-representations which are not equivalent to cell $2$-representations.
\end{abstract}

\maketitle

\section{Introduction and description of the results}\label{s1}

Classical representation theory makes a significant emphasis on problems and techniques
related to the classification of various classes of representations. The recent ``upgrade''
of classical representation theory, known as {\em $2$-representation theory}, has
its abstract origins in \cite{BFK,CR,KL,Ro}. The first classification 
result in $2$-representation theory was obtained in \cite[Subsection~5.4.2]{CR}
where certain ``minimal'' $2$-representations of the Chuang-Rouquier $2$-analogue of
$U(\mathfrak{sl}_2)$ were classified.

The series \cite{MM1,MM2,MM3,MM4,MM5,MM6} of papers initiated a systematic study of 
the so-called {\em finitary} $2$-categories which are natural $2$-analogues of finite
dimensional algebras. The penultimate paper \cite{MM5} of this series proposes the
notion of a {\em simple transitive $2$-representation} which seems to be a natural $2$-analogue
for the classical notion of a simple module. Furthermore, \cite{MM5,MM6} classifies
simple transitive $2$-representations for a certain class of finitary $2$-categories with a
weak involution which enjoy particularly nice combinatorial properties, the so-called
{\em (weakly) fiat} $2$-categories with {\em strongly regular two-sided cells}. Examples of 
the latter $2$-categories include projective functors for finite dimensional self-injective
algebras,  finitary quotients of finite type $2$-Kac-Moody algebras
and Soergel bimodules (over the coinvariant algebra) in type $A$. The classification  results of \cite{MM5,MM6} 
assert that, for such $2$-categories, every simple transitive $2$-representation
is, in fact, equivalent to a so-called {\em cell $2$-representation}, that is a natural
subquotient of the regular (principal) $2$-representation defined combinatorially in \cite{MM1,MM2}.
In \cite{KM1}, this classification was used to describe projective functors on 
parabolic category $\mathcal{O}$ in type $A$.

The problem of classifying simple transitive $2$-representations was recently studied for
several classes of finitary $2$-categories which are not covered by the results in \cite{MM5,MM6}.
In particular, in \cite{Zi} it is shown that cell $2$-representations exhaust 
the simple transitive $2$-representations for the $2$-category of Soergel bimodules in
Weyl type $B_2$. In \cite{MZ}, a similar  result was proved for the $2$-categories of projective
functors for the two smallest non-self-injective finite dimensional algebras. Some other related 
results can be found in \cite{GM1,GM2,Zh1,Zh2}.

An essential novel step in this theory was made in the recent paper \cite{MaMa} where 
a classification of the simple transitive $2$-representations was given for certain 
$2$-sub\-qu\-o\-ti\-ent categories of Soergel bimodules  in two dihedral Coxeter types. 
In one of the cases, it turned out that cell $2$-representations do not 
exhaust the simple transitive $2$-representations. A major part of \cite{MaMa} is devoted to an
explicit construction of the remaining simple transitive $2$-representation. This construction involves 
a subtle interplay of various category theoretic tricks.
 
The present paper explores to which extent the techniques developed in \cite{MM5,Zi,MaMa}
can be used to attack the problem of classification of the simple transitive $2$-representations
for $2$-categories of Soergel bimodules over coinvariant algebras 
in the general case of finite Coxeter systems. We develop the approach and intuition 
described in \cite{Zi,MaMa} further and single out a situation in which this approach seems to
be applicable. The combinatorial structure of the $2$-category of Soergel bimodules is roughly
captured by the Kazhdan-Lusztig combinatorics of the so-called {\em two-sided Kazhdan-Lusztig
cells}. The minimal, with respect to the two-sided order, two-sided cell corresponds to the
identity element. If we take this minimal two-sided cell out, in what remains there is again a unique minimal
two-sided cell. This is the two-sided cell which 
contains all simple reflections. The main object of study in the present
paper is the unique ``simple'' quotient of the $2$-category of Soergel bimodules in which
only these two smallest two-sided cells survive. This is the $2$-category which we call the
{\em small quotient} of the $2$-category of Soergel bimodules.
Our main result is the following statement that combines the statements of  
Theorems~\ref{thm1700}, \ref{thm49}, \ref{thm82},  \ref{thm84},  \ref{thm86} and \ref{thm88}.

\begin{theorem}\label{thmintro}
Let $\underline{\cS}$ be the small  quotient of the $2$-category of Soergel bimodules over the coinvariant algebra
associated to a finite Coxeter system $(W,S)$.
\begin{enumerate}[$($i$)$]
\item\label{thmintro.1} If $W$ has rank greater than two or is of Coxeter type $I_2(n)$, with 
$n=4$ or $n>1$ odd, then 
every simple transitive $2$-representation of $\underline{\cS}$ is equivalent to a cell $2$-rep\-re\-sen\-ta\-ti\-on.
\item\label{thmintro.2} If $W$ is of Coxeter type $I_2(n)$, with $n>4$ even, 
then, apart from cell $2$-representations, $\underline{\cS}$ has two extra equivalence classes of 
simple transitive $2$-representations, see the explicit construction in Subsection~\ref{s6.7}. If 
$n\neq 12,18,30$, these are all the simple transitive $2$-representations.
\end{enumerate}
\end{theorem}

We note that Coxeter type $I_2(4)$ is dealt with  in \cite{Zi} and the result in this case is
similar to Theorem~\ref{thmintro}\eqref{thmintro.1}. In fact, the construction in Subsection~\ref{s6.7}
also works in type $I_2(4)$, however, it results in $2$-representations which turn out to be equivalent
to cell $2$-representations. The exceptional types $I_2(12)$, $I_2(18)$ and $I_2(30)$ exhibit 
some strange connection, which we do not really understand, to Dynkin diagrams of type $E$ that,
at the moment, does not allow us to complete the classification of simple transitive $2$-representations 
in these types.

Theorem~\ref{thmintro} is proved by a case-by-case analysis. Similarly to the general approach of 
\cite{MM5,MM6,Zi,MaMa}, in each case, the proof naturally splits into two major parts:
\begin{itemize}
\item the first part of the proof determines the non-negative integral matrices which represent the
action of Soergel bimodules corresponding to simple reflections;
\item to each particular case of matrices determined in the first part, the second part of the proof 
provides a classification of simple transitive $2$-rep\-re\-sen\-ta\-ti\-ons for which these particular
matrices are realized.
\end{itemize}
Although being technically more involved, the second part of the proof is fairly similar to the
corresponding arguments in \cite{MM5,MM6,Zi,MaMa}. For Theorem~\ref{thmintro}\eqref{thmintro.2},
an additional essential part is the construction of the two extra simple transitive $2$-representations.
This construction follows closely the approach of \cite[Section~5]{MaMa}. Apart from this,
the main general difficulty 
in the proof of Theorem~\ref{thmintro} lies in  the first part of the proof. 
Both the difficulty of this part and the approach we choose to deal with this part depends 
heavily on each particular  Coxeter type we study.

For Coxeter groups of rank higher than two, we use a reduction argument to the rank two case.
For this reduction to work, we need the statement of Theorem~\ref{thmintro}
in Coxeter types $I_2(n)$, where $n=3,4,5$. For $n=3,5$, we may directly refer to 
Theorem~\ref{thmintro}\eqref{thmintro.1}. As already mentioned above, 
the case $n=4$ is treated in \cite{Zi}.
Behind this reduction is the observation that Soergel bimodules corresponding to 
the longest elements in rank two
parabolic subgroups of $W$ do not survive projection onto $\underline{\cS}$. 

The argument for the first part of the proof which we employ in the rank two case is based on
an analysis of spectral properties of some integral matrices. We observe that Fibonacci 
polynomials are connected to minimal polynomials of certain integral matrices which 
encode the combinatorics of simple transitive $2$-representations. Using the 
factorization of Fibonacci polynomials over $\mathbb{Q}$, see \cite{Lev}, and estimating the
values of the maximal possible eigenvalues for the matrices which encode 
the action of Soergel bimodules corresponding to simple reflections in  
simple transitive $2$-representations, allows us to deduce that the only possibility 
for these matrices are the ones appearing in cell $2$-representations.

Apart from Theorem~\ref{thmintro}, we also prove the following 
general result. We refer to Subsection~\ref{s2.4} and \cite[Subsection~3.2]{CM} for the definition of 
the apex and to Subsection~\ref{s2.3} and \cite[Subsection~4.2]{MM2} for 
the definition of the abelianization $\overline{\mathbf{M}}$ of a $2$-representation $\mathbf{M}$.

\begin{theorem}\label{thmmintrotwo}
Let $\mathbf{M}$ be a simple transitive $2$-representation  of a fiat $2$-category $\cC$
with apex $\mathcal{I}$. Then, for every $1$-morphism $\mathrm{F}\in \mathcal{I}$ and
every object $X$ in any $\overline{\mathbf{M}}(\mathtt{i})$, the object $\mathrm{F}\,X$ is projective.
Moreover,  $\overline{\mathbf{M}}(\mathrm{F})$ is a projective functor.
\end{theorem}

Our proof of this statement is based crucially on the results from \cite{KM2}. 
Theorem~\ref{thmmintrotwo} will be very useful in further studies of simple transitive $2$-representations. 
It applies directly to all cases we consider and substantially simplifies some of the arguments. 
For example, in Coxeter type $I_2(5)$, we originally had an independent
argument for a similar result which involved a very technical statement that a certain collection of 
twenty seven linear inequalities is equivalent to the fact that all these inequalities are, in fact,
equalities. This full argument was three pages long and just covered one Coxeter type. A similar
argument in other Coxeter types of the form $I_2(n)$ seemed, for a long period of time, unrealistic.

Our results form a first step towards classification of simple transitive 
$2$-rep\-re\-sen\-tations of Soergel bimodules in all types. However, for the moment, the 
technical difficulty of solving the first part of the problem (as described above) 
in the general case seems too high. Already in type $B_3$ the classification of
simple transitive $2$-representations is not complete. It is also of course very 
natural to ask what happens in positive characteristics. However, the situation there is
expected to be even more complicated.

The paper is organized as follows: In Section~\ref{s2} we collect all necessary preliminaries
on $2$-categories and $2$-representations. Section~\ref{s3} collects preliminaries on 
$2$-categories of Soergel bimodules. In Section~\ref{s4} we collect several general results
related to classification of simple transitive $2$-representations of small quotients of
Soergel bimodules. It also contains our proof of Theorem~\ref{thmmintrotwo}, see
Theorem~\ref{thmminimal} in Subsection~\ref{s5.5}. 
Section~\ref{sim} contains auxiliary results on some spectral properties of integral matrices.
Coxeter type  $I_2(n)$, for $n$ odd, is studied
in Section~\ref{s17}. Coxeter type  $I_2(n)$, for $n$ even, is studied in Section~\ref{s6}.
Section~\ref{s7} collects our study of all non-simply laced Coxeter types of rank higher than two
(all simply laced types are covered by the results of \cite{MM5,MM6}). Finally, in 
Section~\ref{s8} we propose a new general construction of finitary $2$-categories. 
The novel component of this construction is that indecomposable $2$-morphisms are,
in general, defined using decomposable functors. This allows us to give an alternative
description of one interesting example of a finitary $2$-category  from \cite[Example~8]{Xa}.
\vspace{0.5cm}

{\bf Warning.} All Soergel bimodules considered in this paper are over the coinvariant algebra,
not the polynomial algebra. Furthermore, our setup is ungraded.
\vspace{0.5cm}

{\bf Acknowledgment.} 
This research was partially supported by the Swedish Research Council (for V.~M.),
Knut and Alice Wallenbergs Foundation (for T.~K. and V.~M.) and G{\"o}ran Gustafsson Foundation (for V.~M.). 
We thank Vanessa Miemietz and Xiaoting Zhang for poining out a gap in one of the proofs. 
We thank the referee for helpful comments.
\vspace{5mm}

\section{$2$-categories and $2$-representations}\label{s2}

\subsection{Notation and conventions}\label{s2.1}
 
We work over the field $\mathbb{C}$ of complex numbers 
and denote  $\otimes_{\mathbb{C}}$ by $\otimes$. 
A module always means a {\em left} module. 
All maps are composed from right to left.
 
\subsection{Finitary and fiat $2$-categories}\label{s2.2}

For generalities on $2$-categories, we refer the reader to \cite{Le,Mc,Ma}. 

By a $2$-category we mean a category enriched over the monoidal category 
$\mathbf{Cat}$ of small categories. In other words, a $2$-category $\cC$ consists 
of objects (which we denote by Roman lower case letters in a typewriter font), 
$1$-morphisms (which we denote by capital Roman letters), and  $2$-morphisms 
(which we denote by Greek lower case letters), composition of  $1$-morphisms, 
horizontal and  vertical compositions of $2$-morphisms (denoted $\circ_0$ and 
$\circ_1$, respectively), identity $1$-morphisms and identity $2$-morphisms. 
These satisfy the obvious collection of axioms. 

For a $1$-morphism $\mathrm{F}$, 
we denote by $\mathrm{id}_{\mathrm{F}}$ the corresponding identity $2$-morphism. 
We often write $\mathrm{F}(\alpha)$ for $\mathrm{id}_{\mathrm{F}}\circ_0\alpha$
and $\alpha_{\mathrm{F}}$ for $\alpha\circ_0\mathrm{id}_{\mathrm{F}}$.

We say that a $2$-category $\cC$ is {\em finitary} if each category 
$\cC(\mathtt{i},\mathtt{j})$ is an idempotent split, additive and Krull-Schmidt 
$\mathbb{C}$-linear category with finitely many isomorphism classes of indecomposable 
objects and finite dimensional morphism spaces, moreover, all compositions are assumed to
be compatible with these additional structures, see \cite[Subsection~2.2]{MM1} for details.

If $\cC$ is a  finitary $2$-category, we say that $\cC$ is {\em fiat} provided that 
it has a weak involution $\star$ together with adjunction $2$-morphisms satisfying 
the axioms of adjoint functors, for each pair $(\mathrm{F},\mathrm{F}^{\star})$ of 
$1$-morphisms, see \cite[Subsection~2.4]{MM1} for details. Similarly, we say that $\cC$ is 
{\em weakly fiat} provided that it has a weak anti-autoequivalence $\star$  
(not necessarily involutive)  together with adjunction $2$-morphisms satisfying 
the axioms of adjoint functors, for each pair $(\mathrm{F},\mathrm{F}^{\star})$ of 
$1$-morphisms, see \cite[Subsection~2.5]{MM6} for details. 

\subsection{$2$-representations}\label{s2.3}

Let $\cC$ be a finitary $2$-category. The $2$-category $\cC$-afmod consists of all 
{\em finitary $2$-representations} of $\cC$ as defined in \cite[Subsection~2.3]{MM3}. Objects in
$\cC$-afmod are strict functorial actions of $\cC$ on idempotent split, additive and 
Krull-Schmidt $\mathbb{C}$-linear categories which have finitely many isomorphism 
classes of indecomposable objects and finite dimensional spaces of morphisms. 
In $\cC$-afmod, $1$-morphisms are strong $2$-natural transformations and $2$-morphisms 
are modifications.

Similarly, we consider the $2$-category $\cC$-mod consisting of all {\em abelian 
$2$-rep\-re\-sen\-ta\-ti\-ons} of $\cC$. These are functorial actions of $\cC$ on 
categories equivalent to module categories over finite dimensional algebras, 
we again refer to \cite[Subsection~2.3]{MM3} for details. The $2$-categories $\cC$-afmod
and $\cC$-mod are connected by the diagrammatically defined {\em abelianization $2$-functor}
\begin{displaymath}
\overline{\hspace{1mm}\cdot\hspace{1mm}}:
\cC\text{-}\mathrm{afmod}\to \cC\text{-}\mathrm{mod},
\end{displaymath}
see \cite[Subsection~4.2]{MM2} for details. For $\mathbf{M}\in \cC\text{-}\mathrm{afmod}$,
objects in $\overline{\mathbf{M}}$ are diagrams of the form $X\longrightarrow Y$ over $\mathbf{M}(\mathtt{i})$'s
and morphisms are quotients of the space of (solid) commutative squares of the form
\begin{displaymath}
\xymatrix{ 
X\ar[rr]\ar[d]&&Y\ar[d]\ar@{.>}[dll]\\
X\ar[rr]&&Y'
}
\end{displaymath}
modulo the subspace for which the right horizontal map factorizes via some dotted map. 
The action of $\cC$ is defined component-wise.

We say that two $2$-representations are  {\em equivalent} provided that there is 
a strong $2$-natural transformation between them which restricts to 
an equivalence of categories, for each object  in $\cC$.

A finitary $2$-representation $\mathbf{M}$ of $\cC$ is said to be {\em transitive} provided 
that, for any indecomposable objects $X$ and $Y$ in $\displaystyle 
\coprod_{\mathtt{i}\in\ccC}\mathbf{M}(\mathtt{i})$, there is a $1$-morphism $\mathrm{F}$ in $\cC$ 
such that the object $Y$ is isomorphic to a direct summand of the object $\mathbf{M}(\mathrm{F})\, X$. 
A transitive $2$-representation $\mathbf{M}$  is said to be {\em simple transitive } provided that 
$\displaystyle \coprod_{\mathtt{i}\in\ccC}\mathbf{M}(\mathtt{i})$ 
does not have any non-zero proper ideals which are invariant under the functorial  action of $\cC$.
Given a finitary $2$-representation $\mathbf{M}$ of $\cC$, the {\em rank} of 
$\mathbf{M}$ is the number of isomorphism classes of indecomposable objects in 
\begin{displaymath}
\coprod_{\mathtt{i}\in\ccC}\mathbf{M}(\mathtt{i}). 
\end{displaymath}

For simplicity, we will often use the ``module'' notation $\mathrm{F}\, X$ 
instead of the corresponding ``representation'' notation $\mathbf{M}(\mathrm{F})\, X$.

\subsection{Combinatorics}\label{s2.4}

Let $\cC$ be a finitary $2$-category and $\mathcal{S}[\cC]$ the corresponding 
{\em multisemigroup} in the sense of \cite[Section~3]{MM2}. 
The objects in $\mathcal{S}[\cC]$ are isomorphism classes of indecomposable 
$1$-morphisms in $\cC$. For $\mathrm{F},\mathrm{G}\in \mathcal{S}[\cC]$, we set
$\mathrm{F}\leq_L\mathrm{G}$ provided that $\mathrm{G}$ is isomorphic to 
a summand of $\mathrm{H}\circ \mathrm{F}$, for some $1$-morphism $\mathrm{H}$.
The relation $\leq_L$ is called the {\em left} (pre)-order on  $\mathcal{S}[\cC]$.
The  {\em right} and {\em two-sided} (pre)-orders are defined similarly using
multiplication on the right or on both sides.
Equivalence classes with respect to these orders 
are called {\em cells} and the corresponding equivalence relations are denoted
$\sim_L$,  $\sim_R$ and $\sim_J$, respectively.  As usual, for simplicity, 
we say ``cells of $\cC$'' instead of  ``cells of $\mathcal{S}[\cC]$''.

Given a two-sided cell $\mathcal{J}$ in $\cC$, we call $\cC$
{\em $\mathcal{J}$-simple} provided that every non-zero two-sided $2$-ideal in 
$\cC$ contains the identity $2$-morphisms for some (and hence for all) 
$1$-morphisms in $\mathcal{J}$, see \cite[Subsection~6.2]{MM2} for details. 

A two-sided cell is called {\em strongly regular}, see \cite[Subsection~4.8]{MM1}, provided that 
\begin{itemize}
\item different left (resp. right) cells inside $\mathcal{J}$ are not comparable with 
respect to $\leq_L$ (resp. $\leq_R$);
\item the intersection of any left and any right cell inside $\mathcal{J}$ consists of exactly
one element.
\end{itemize}
By \cite[Corollary~19]{KM2}, the first condition is automatically satisfied for all fiat $2$-categories.

By \cite[Subsection~3.2]{CM}, each simple transitive $2$-representation $\mathbf{M}$ of $\cC$
has an {\em apex}, which is defined as a unique two-sided cell that is maximal in the set of 
all two-sided cells whose elements are not annihilated by $\mathbf{M}$.

\subsection{Cell $2$-representations}\label{s2.5}

For every $\mathtt{i}\in \cC$, we will denote by $\mathbf{P}_{\mathtt{i}}:=\cC_A(\mathtt{i},{}_-)$
the corresponding {\em principal} $2$-representation. For a left cell $\mathcal{L}$
in $\cC$, there is  $\mathtt{i}\in \cC$ such that all $1$-morphisms in $\mathcal{L}$ start at $\mathtt{i}$.
Let $\mathbf{N}_{\mathcal{L}}$ denote the $2$-subrepresentation of $\mathbf{P}_{\mathtt{i}}$
given by the additive closure of all $1$-morphisms $\mathbf{F}$ satisfying $\mathbf{F}\geq_L \mathcal{L}$.
Then $\mathbf{N}_{\mathcal{L}}$ has a unique maximal $\cC$-invariant ideal and the corresponding quotient
$\mathbf{C}_{\mathcal{L}}$ is called the {\em cell} $2$-representation associated to $\mathcal{L}$.
We refer the reader to  \cite[Section~4]{MM1} and \cite[Subsection~6.5]{MM2} for more details.

\subsection{Matrices in the Grothendieck group}\label{s2.6}
         
Let $\cC$ be a finitary $2$-category and $\mathbf{M}$ a finitary $2$-representation of $\cC$.
Fix a complete and irredundant list of representatives in all isomorphism classes of 
indecomposable objects in $\displaystyle \coprod_{\mathtt{i}}\mathbf{M}(\mathtt{i})$.
Then, for every $1$-morphism $\mathrm{F}$ in $\cC$, we have the corresponding matrix 
$\Lparen \mathrm{F}\Rparen$ which describes multiplicities in direct sum decompositions 
of the images of indecomposable objects under $\mathbf{M}(\mathrm{F})$.

If $\cC$ is fiat, then each $\overline{\mathbf{M}}(\mathrm{F})$ is exact and we also 
have the matrix $\llbracket \mathrm{F}\rrbracket$ which describes composition multiplicities 
of the images of simple objects under $\overline{\mathbf{M}}(\mathrm{F})$. 
By adjunction, the matrix $\llbracket \mathrm{F}^{\star}\rrbracket$ is transposed
to the matrix $\Lparen \mathrm{F}\Rparen$.

\subsection{Based modules over positively based algebras}\label{s2.7}
         
In this subsection we recall some notation and results from \cite{KM2}.
Let $A$ be a finite dimensional $\mathbb{C}$-algebra with a fixed basis $\{a_1,a_2,\dots,a_n\}$.
Assume that $A$ is {\em positively based} in the sense that all structure constants $\gamma_{i,j}^s$
with respect to this basis, defined via
\begin{displaymath}
a_ia_j=\sum_{s=1}^n\gamma_{i,j}^sa_s,\quad\text{ where } \gamma_{i,j}^s\in\mathbb{C},
\end{displaymath}
are non-negative real numbers. For $s,j\in\{1,2,\dots,n\}$, we set $a_s\geq_L a_j$
provided that $\gamma_{i,j}^s\neq 0$ for some $i$. We set $a_s\sim_L a_j$ provided that 
$a_s\geq_L a_j$ and $a_j\geq_L a_s$ at the same time. This defines the {\em left order} and
the corresponding {\em left cells}. The right and two-sided orders and the right and two-sided cells
are defined similarly (using multiplication from the right or from both sides) and denoted 
$\geq_R$, $\geq_J$, $\sim_R$ and $\sim_J$, respectively. 

For each left cell $\mathcal{L}$, we have the corresponding {\em left cell} $A$-module
$C_{\mathcal{L}}$. The module $C_{\mathcal{L}}$ contains a subquotient $L_{\mathcal{L}}$, called
the {\em special subquotient}, which has composition multiplicity one in $C_{\mathcal{L}}$.
This subquotient is defined as the unique subquotient of $C_{\mathcal{L}}$ which contains the
Perron-Frobenius eigenvector of $C_{\mathcal{L}}$ for the linear operator $\sum_i a_i$.
We have $L_{\mathcal{L}}\cong L_{\mathcal{L}'}$ if $\mathcal{L}$ and $\mathcal{L}'$ belong to
the same two-sided cell. The latter justifies the notation $L_{\mathcal{J}}:=L_{\mathcal{L}}$, where
$\mathcal{J}$ is the two-sided cell containing $\mathcal{L}$.

An $A$-module $V$ with a fixed basis $\{v_1,v_2,\dots,v_k\}$ is called {\em positively based} provided
that each $a_i\cdot v_j$ is a linear combination of the $v_m$'s with non-negative real coefficients.
A positively based module $V$ is called {\em transitive} provided that, for any $v_i$ and $v_j$, there
is an $a_s$ such that $v_j$ appears in $a_s\cdot v_i$ with a non-zero coefficient. For each transitive
$A$-module $V$, there is a unique two-sided cell $\mathcal{J}(V)$ which is the maximum element
(with respect to the two-sided order) in the set of all two-sided cells whose elements 
do not annihilate $V$. The two-sided cell $\mathcal{J}(V)$ is called the {\em apex} of $V$.
The two-sided cell $\mathcal{J}(V)$ is {\em idempotent} in the sense that it contains some
$a_i$, $a_j$ and $a_s$ (non necessarily different) 
such that $\gamma_{i,j}^s\neq 0$. The simple subquotient $L_{\mathcal{J}(V)}$
has composition multiplicity one in $V$ and is also called the {\em special} subquotient of $V$.

\subsection{Decategorification}\label{s2.8}

For a finitary $2$-category $\cC$, we denote by $A_{\ccC}$ the complexification of the Grothendieck
decategorification of $\cC$. The algebra $A_{\ccC}$ is positively based with respect to the {\em natural basis}
given by the classes of indecomposable $1$-morphisms in $\cC$, see \cite[Subsection~1.2]{Ma} or 
\cite[Subsection~2.5]{KM2} for details.

Given a $2$-representation $\mathbf{M}$ of $\cC$, the Grothendieck decategorification of 
$\mathbf{M}$ is, naturally, a positively based $A_{\ccC}$-module. Moreover, the latter is transitive if
and only if $\mathbf{M}$ is transitive. This allows us to speak about the {\em apex} of $\mathbf{M}$
in the obvious way, see \cite[Subsection~3.2]{CM} for details.

We denote by $\mathrm{F}_{\mathrm{tot}}$ a multiplicity free direct sum of all indecomposable
$1$-morphisms in $\cC$. Given a $2$-representation $\mathbf{M}$ of $\cC$, we set 
$\mathtt{M}_{\mathrm{tot}}:=\Lparen \mathrm{F}_{\mathrm{tot}}\Rparen$. Then $\mathbf{M}$ is transitive 
if and only if all coefficients of $\mathtt{M}_{\mathrm{tot}}$ are non-zero (and thus positive integers).

\section{Soergel bimodules}\label{s3}

\subsection{Soergel bimodules for finite Coxeter groups}\label{s3.1}

Let $(W,S)$ be a finite irreducible Coxeter system and $\varphi:W\to\mathrm{GL}(V)$ be the geometric 
representation of $W$ as in \cite[Section~5.3]{Hu}. Here $V$ is a real vector space.
We denote by $\leq$ the Bruhat order and by $\mathbf{l}:W\to\mathbb{Z}_{\geq 0}$
the length function. For $w\in W$, we denote by $\underline{w}$ the corresponding element in
the Kazhdan-Lusztig basis of $\mathbb{Z}[W]$, see \cite{KaLu,So2,EW}. Each $\underline{w}$ is a 
linear combination of elements in $W$ with non-negative integer coefficients.

Let $\mathtt{C}$ be the (complexified) coinvariant algebra associated to $(W,S,V)$.
For $s\in S$, we denote by $\mathtt{C}^s$ the subalgebra of 
$s$-invariants in $\mathtt{C}$. A {\em Soergel $\mathtt{C}\text{-}\mathtt{C}$--bimodule} is a 
$\mathtt{C}\text{-}\mathtt{C}$--bimodule isomorphic to a bimodule from the additive closure of the
monoidal category of  $\mathtt{C}\text{-}\mathtt{C}$--bimodules generated by 
$\mathtt{C}\otimes_{\mathtt{C}^s}\mathtt{C}$, where $s$ runs through $S$, see \cite{So,So2,Li}.
We also note that by the {\em additive closure} of some $X$ we mean the the full subcategory whose
objects are isomorphic to finite direct sums of direct summands of $X$.
Isomorphism classes of indecomposable Soergel bimodules are naturally indexed by $w\in W$,
we denote by $B_w$ a fixed representative from such a class. There is an isomorphism between $\mathbb{Z}[W]$
and the Grothendieck ring of the monoidal category of Soergel bimodules for $(W,S,V)$.
This isomorphism sends $\underline{w}$ to the class of $B_w$, see \cite{So2,E,EW}. 

Let $\mathcal{C}$ be a small category equivalent to the category $\mathtt{C}\text{-}\mathrm{mod}$. Define
the $2$-category $\cS=\cS(W,S,V,\mathcal{C})$ of {\em Soergel bimodules} associated to 
$(W,S,V,\mathcal{C})$ as follows:
\begin{itemize}
\item $\cS$ has one object $\mathtt{i}$, which we can identify with $\mathcal{C}$;
\item $1$-morphisms in $\cS$ are all endofunctors of $\mathcal{C}$ which are isomorphic to endofunctors
given by tensoring with Soergel $\mathtt{C}\text{-}\mathtt{C}$--bimodules;
\item $2$-morphisms in $\cS$ are natural transformations of functors
(these correspond to homomorphisms of Soergel $\mathtt{C}\text{-}\mathtt{C}$--bimodules).
\end{itemize}
The $2$-category $\cS$ is fiat. The algebra $A_{\ccS}$ is  isomorphic to $\mathbb{C}[W]$ 
and is positively based with respect to the  Kazhdan-Lusztig basis in $\mathbb{C}[W]$.
For $w\in W$, let $\theta_w$ denote a fixed representative in the isomorphism class of 
indecomposable $1$-morphisms in $\cS$ given by tensoring with $B_w$. 

\subsection{Small quotients of Soergel bimodules}\label{s3.2}

Let $(W,S)$, $V$, $\mathcal{C}$ and $\cS$ be as above. 

\begin{lemma}\label{lem1}
For any $s,t\in S$, we have $\theta_s\sim_J\theta_t$. 
\end{lemma}

\begin{proof}
As $(W,S)$ is irreducible, we may assume that $st\neq ts$. 
As $\theta_s\theta_t\cong \theta_{st}$, we have $\theta_s\leq_R\theta_{st}$.
Further, as $st\neq ts$, we have $\theta_{st}\theta_s\cong \theta_{sts}\oplus \theta_s$.
This implies that $\theta_s\geq_R\theta_{st}$ and hence $\theta_s\sim_R\theta_{st}$.
Similarly one shows that $\theta_t\sim_L\theta_{st}$.
The claim follows.
\end{proof}

After Lemma~\ref{lem1}, we may define the two-sided cell $\mathcal{J}$ of $\cS$ as the
two-sided cell containing $\theta_s$, for all $s\in S$. As $\underline{s}$, where 
$s\in S$, generate $\mathbb{Z}[W]$, it follows that $\mathcal{J}$ is the unique minimal
element, with respect to the two-sided order, in the set of all two-sided cells of $\cS$
different from the two-sided cell corresponding to $\theta_e$.

Let $\cI$ be the unique $2$-ideal of $\cS$ which is maximal with respect to the property that
it does not contain any $\mathrm{id}_{\mathrm{F}}$, where $\mathrm{F}\in \mathcal{J}$.
The $2$-category $\cS/\cI$ will be called the {\em small quotient} of $\cS$ and denoted
by $\underline{\cS}$. By construction, the $2$-category $\underline{\cS}$ is fiat and
$\mathcal{J}$-simple. It has two two-sided cells, namely, the two-sided cell 
corresponding to $\theta_e$ and the image of $\mathcal{J}$, which we identify with 
$\mathcal{J}$, abusing notation.

\section{Generalities on simple transitive $2$-representations of $\underline{\cS}$}\label{s4}

\subsection{Basic combinatorics of $\mathcal{J}$}\label{s4.1}
Here we describe the basics of the Kazhdan-Lusztig combinatorics related to the two-sided cell
$\mathcal{J}$. Recall Lusztig's $\mathbf{a}$-function $\mathbf{a}:W\to \mathbb{Z}_{\geq 0}$,
which is defined in  \cite{Lu1}. This function has, in particular, the following properties: 
it is constant on two-sided cells in $W$, and, moreover, we have the equality $\mathbf{a}(w)=\mathbf{l}(w)$ 
provided that $w$ is the longest element in some parabolic subgroup of $W$.

\begin{proposition}\label{propJKL}
{\hspace{2mm}} 
 
\begin{enumerate}[$($i$)$]
\item\label{propJKL.1} The map $\mathcal{L}\mapsto \mathcal{L}\cap S$ is a bijection between the set of
left cells in $\mathcal{J}$ and $S$.
\item\label{propJKL.2} For any left cell $\mathcal{L}$ in $\mathcal{J}$, the unique element in $\mathcal{L}\cap S$ 
is the Duflo involution in $\mathcal{L}$.
\item\label{propJKL.3} The map $\mathcal{R}\mapsto \mathcal{R}\cap S$ is a bijection between the set of
right cells in $\mathcal{J}$ and $S$.
\item\label{propJKL.4} For any right cell $\mathcal{R}$ in $\mathcal{J}$, the unique element in $\mathcal{R}\cap S$ 
is the Duflo involution in $\mathcal{R}$.
\item\label{propJKL.5} An element $w\in W$ such that $w\neq e$ belongs to $\mathcal{J}$ if and only if $w$
has a unique reduced expression. 
\item\label{propJKL.6} An element $w\in W$ belongs to $\mathcal{J}$ if and only if $\mathbf{a}(w)=1$. 
\end{enumerate}
\end{proposition}

If $W$ is a Weyl group, all these results are mentioned in \cite{Do} with references to \cite{Lu1,Lu2,Lu3}.
We note the difference in both the normalizations of the Hecke algebra 
and the choice of the Kazhdan-Lusztig basis in this paper and in \cite{Do}. 
Below, we outline a general argument.

\begin{proof}
We start by observing that different simple reflections have different left (right) descent sets
and hence are not right (left) equivalent. 

Further, if $w\in W$ has more than one reduced expression, then any such reduced expression 
can be obtained from any other by means of the braid relations, see \cite[\S IV.1.5]{Bo}.
This means that, with respect to the Bruhat order, $w$ is bigger
than or equal to the longest element in some parabolic subgroup of $W$ of rank two. 
As longest element in parabolic subgroup of $W$ of rank two are not in $\mathcal{J}$
(since the value of $\mathbf{a}$ on such elements is strictly bigger than $1=\mathbf{a}(s)$, 
where $s\in S$), we get $w\not\in\mathcal{J}$. 

By induction with respect to the rank of $W$, one checks that any two elements 
of the form $xs$ and $ys$ (resp. $sx$ and $sy$) which, moreover,  have unique reduced expressions, 
belong to the same left (resp. right) cell.
Indeed, the basis of the induction is the rank two case which is standard. The induction step
follows from the relation $\theta_s\theta_t\theta_s=\theta_{sts}\oplus \theta_s$ which 
is true for any pair $s$, $t$ of non-commuting simple reflections.
This, combined with the observation in the previous paragraph, 
yields \eqref{propJKL.1}, \eqref{propJKL.3}, \eqref{propJKL.5} and \eqref{propJKL.6}. 

That simple reflections are Duflo involutions in their cells follows
directly from the definitions since the value of $\mathbf{a}$ on $\mathcal{J}$ is $1$, 
implying \eqref{propJKL.2}, \eqref{propJKL.4}. This completes the proof.
\end{proof}

An important consequence of Proposition~\ref{propJKL}\eqref{propJKL.5} is the following: 
if $s$ and $t$ are different elements in $S$,
then the longest element in the parabolic subgroup of $W$ generated by $s$ and $t$ has two different 
reduced expressions and thus does not belong to $\mathcal{J}$.
We also refer the reader to \cite{De,BW} for related results.

\subsection{Examples and special cases}\label{s4.2}

Using Proposition~\ref{propJKL}, one can describe elements in $\mathcal{J}$ explicitly as products
of simple reflections. Here we list several examples and special cases under the following conventions:
\begin{itemize}
\item we provide a picture of the Coxeter diagram followed by the list of elements in $\mathcal{J}$
organized in a square table, each element is given by its unique reduced expression;
\item vertices of the Coxeter diagram are numbered by positive integers and we write $i$ for the
corresponding simple reflection $s_i$;
\item the left cell $\mathcal{L}_i$ containing $i$ is the $i$-th column of the diagram;
\item the right cell $\mathcal{R}_i$ containing $i$ is the $i$-th row of the diagram; 
\end{itemize}
To verify these examples one could use the following general approach:
\begin{itemize}
\item check that all elements given in the example have a unique reduced expression;
\item check that, making any step outside the list in the example, gives an element 
with more than one reduced expression;
\item note that all elements in a left cell start with the same simple reflection on the right,
and all elements in a right cell start with the same simple reflection on the left;
\item use Proposition~\ref{propJKL}.
\end{itemize}

We start with types $A_3$ and $D_4$.
\begin{displaymath}
\xymatrix{1\ar@{-}[r]&2\ar@{-}[r]&3}\qquad \qquad 
\begin{array}{c||c|c|c}
&\mathcal{L}_1&\mathcal{L}_2&\mathcal{L}_3\\
\hline\hline
\mathcal{R}_1&1&12&123\\
\hline
\mathcal{R}_2&21&2&23\\
\hline
\mathcal{R}_3&321&32&3
\end{array}
\end{displaymath}
\begin{displaymath}
\xymatrix{&4&\\1\ar@{-}[r]&2\ar@{-}[r]\ar@{-}[u]&3}\qquad \qquad 
\begin{array}{c||c|c|c|c}
&\mathcal{L}_1&\mathcal{L}_2&\mathcal{L}_3&\mathcal{L}_4\\
\hline\hline
\mathcal{R}_1&1&12&123&124\\
\hline
\mathcal{R}_2&21&2&23&24\\
\hline
\mathcal{R}_3&321&32&3&324\\
\hline
\mathcal{R}_4&421&42&423&4
\end{array}
\end{displaymath}

Similarly to the above, we have the following observation which follows directly from 
Proposition~\ref{propJKL}.

\begin{corollary}\label{corsimplylaced}
If the Coxeter diagram of $W$ is simply laced, then the cell $\mathcal{J}$ is strongly regular. Namely, 
the intersection $\mathcal{L}_i\cap \mathcal{R}_j$ consists of $s_js_{i_k}\dots s_{i_2}s_{i_1}s_i$, where 
$j-i_k-\dots- i_1-i$ is the unique path in the diagram connecting $i$ and $j$.
\end{corollary}

Because of Corollary~\ref{corsimplylaced}, in the simply laced case, any simple transitive 
$2$-rep\-re\-sen\-ta\-ti\-on of $\underline{\cS}$ is equivalent to a cell $2$-representation by 
\cite[Theorem~18]{MM5}. We note that the original formulation of 
\cite[Theorem~18]{MM5} has an additional assumption, the 
so-called {\em numerical condition} which appeared in \cite[Formula~(10)]{MM1}. 
This additional assumption is rendered superfluous by \cite[Proposition~1]{MM6}. 

Consequently, 
\begin{center}
\bf in what follows, we assume that $W$ is not simply laced.
\end{center}

In type  $G_2$, we have the following.
\begin{displaymath}
\xymatrix{1\ar@{-}[r]^{6}&2}\qquad \qquad 
\begin{array}{c||c|c}
&\mathcal{L}_1&\mathcal{L}_2\\
\hline\hline
\mathcal{R}_1&1,121,12121&12,1212\\
\hline
\mathcal{R}_2&21,2121&2,212,21212
\end{array}
\end{displaymath}

The general type $B_n$, for $n\geq 2$, corresponds to the Coxeter diagram
\begin{displaymath}
\xymatrix{1\ar@{-}[r]^{4}&2\ar@{-}[r]&3\ar@{-}[r]&\dots\ar@{-}[r]&n} 
\end{displaymath}
and is given below:
\begin{displaymath}
\begin{array}{c||c|c|c|c|c}
&\mathcal{L}_1&\mathcal{L}_2&\mathcal{L}_3&\dots&\mathcal{L}_n\\
\hline\hline
\mathcal{R}_1&1,121&12&123&\dots&12\cdots n \\
\hline
\mathcal{R}_2&21&2,212&23, 2123&\dots&23\cdots n,212\cdots n \\
\hline
\mathcal{R}_3&321&32,3212&3,32123&\dots&34\cdots n,3212\cdots n \\
\hline
\vdots&\vdots&\vdots&\vdots&\ddots&\vdots \\
\hline
\mathcal{R}_n&n\cdots 21&n\cdots 32,n\cdots212&n\cdots43,n\cdots 2123&\dots&n,n\cdots212\cdots n 
\end{array}
\end{displaymath}

The remaining Weyl type $F_4$ looks as follows:
\begin{displaymath}
\xymatrix{1\ar@{-}[r]&2\ar@{-}[r]^{4}&3\ar@{-}[r]&4}\qquad \qquad 
\begin{array}{c||c|c|c|c}
&\mathcal{L}_1&\mathcal{L}_2&\mathcal{L}_3&\mathcal{L}_4\\
\hline\hline
\mathcal{R}_1&1,12321&12,1232&123&1234\\
\hline
\mathcal{R}_2&21,2321&2,232&23&234\\
\hline
\mathcal{R}_3&321&32&3,323&34,3234\\
\hline
\mathcal{R}_4&4321&432&43,4323&4,43234\\
\end{array}
\end{displaymath}

The exceptional Coxeter type $H_3$ is the following:
\begin{displaymath}
\xymatrix{1\ar@{-}[r]^{5}&2\ar@{-}[r]&3}\qquad \qquad 
\begin{array}{c||c|c|c}
&\mathcal{L}_1&\mathcal{L}_2&\mathcal{L}_3\\
\hline\hline
\mathcal{R}_1&1,121&12,1212&123,12123\\
\hline
\mathcal{R}_2&21,2121&2,212&23,2123\\
\hline
\mathcal{R}_3&321,32121&32,3212&3,32123
\end{array}
\end{displaymath}

The exceptional Coxeter type $H_4$ has the Coxeter diagram
\begin{displaymath}
\xymatrix{1\ar@{-}[r]^{5}&2\ar@{-}[r]&3\ar@{-}[r]&4}
\end{displaymath}
and the corresponding table looks as follows:
\begin{displaymath}
\begin{array}{c||c|c|c|c}
&\mathcal{L}_1&\mathcal{L}_2&\mathcal{L}_3&\mathcal{L}_4\\
\hline\hline
\mathcal{R}_1&1,121&12,1212&123,12123&1234,121234\\
\hline
\mathcal{R}_2&21,2121&2,212&23,2123&234,21234\\
\hline
\mathcal{R}_3&321,32121&32,3212&3,32123&34,321234\\
\hline
\mathcal{R}_4&4321,432121&432,43212&43,432123&4,4321234
\end{array}
\end{displaymath}

The general dihedral type with Coxeter diagram
\begin{displaymath}
\xymatrix{1\ar@{-}[r]^{k}&2},
\end{displaymath}
where $k\geq 3$, looks as follows:
\begin{displaymath}
\begin{array}{c||c|c}
&\mathcal{L}_1&\mathcal{L}_2\\
\hline\hline
\mathcal{R}_1&1,121,\dots,12\cdots 21&12,1212,12\cdots 12\\
\hline
\mathcal{R}_2&21,2121,\dots, 21\cdots 21&2,212,\dots,21\cdots 12
\end{array}
\end{displaymath}
Note that the length of all elements in the latter table is strictly less than $k$.
In particular, the diagonal boxes always contain $\lfloor\frac{k}{2}\rfloor$ elements
while the off-diagonal boxes contain $\lfloor\frac{k}{2}\rfloor$ elements if $k$ is odd
and $\lfloor\frac{k}{2}\rfloor-1$ elements if $k$ is even.

\subsection{The principal element in $\mathbb{Z}[W]$}\label{s4.4}

Following \cite{Zi}, we consider the element 
\begin{displaymath}
\mathbf{s}:=\sum_{s\in S} \underline{s}\in \mathbb{Z}[W],
\end{displaymath}
which we call the {\em principal} element. This element is the decategorification of 
\begin{displaymath}
\mathrm{F}_{\mathrm{pr}}:=\bigoplus_{s\in S} \theta_s.
\end{displaymath}
The main aim of this subsection is to determine the form of the matrix
\begin{displaymath}
\mathtt{M}=\mathtt{M}_{\mathrm{pr}}:=\Lparen \mathrm{F}_{\mathrm{pr}}\Rparen. 
\end{displaymath}
We start by recalling the following.

\begin{lemma}\label{lem2}
Let $\mathbf{M}$ be a transitive $2$-representation of $\underline{\cS}$ and $s\in S$. 
Then there is an ordering of indecomposable objects in  $\mathbf{M}(\mathtt{i})$ such that 
\begin{displaymath}
\Lparen \theta_s\Rparen=
\left(\begin{array}{c|c}2E&*\\\hline 0&0\end{array}\right),
\end{displaymath}
where $E$ is the identity matrix.
\end{lemma}

\begin{proof}
Mutatis mutandis the proof of \cite[Lemma~6.4]{Zi}. 
\end{proof}

\begin{lemma}\label{lem3}
Let $\mathbf{M}$ be a transitive $2$-representation of $\underline{\cS}$ with apex $\mathcal{J}$
and $P$ an indecomposable object in $\mathbf{M}(\mathtt{i})$. Then there exists a unique
$s\in S$ having the property  $\theta_s\, P\cong P\oplus P$.
\end{lemma}

\begin{proof}
Assume that $\theta_s\, P\not \cong P\oplus P$ for every $s\in S$. From Lemma~\ref{lem2} it then
follows that $\mathtt{M}$ has a zero row. Therefore the same row will be zero in any power of $\mathtt{M}$. Note that 
any indecomposable $1$-morphism of $\underline{\cS}$ which is not isomorphic to $\theta_e$
appears as a direct summand of some power of $\mathrm{F}_{\mathrm{pr}}$. If $\mathbf{M}$ had 
rank one, it would follow that all indecomposable $1$-morphisms of $\underline{\cS}$ which 
are not isomorphic to $\theta_e$ annihilate $\mathbf{M}$. This contradicts our assumption that 
$\mathcal{J}$ is the apex of  $\mathbf{M}$. Therefore $\mathbf{M}$ has rank at least two. 
However, in this case, the same row which is zero in all powers of $\mathtt{M}$ must have a zero entry 
in $\mathtt{M}_{\mathrm{tot}}$. This contradicts transitivity of $\mathbf{M}$ and hence establishes
existence of $s\in S$ such that  $\theta_s\, P\cong P\oplus P$.

Assume that $s$ and $t$ are two different elements in $S$ having the property  that 
$\theta_s\, P\cong P\oplus P$ and $\theta_t\, P\cong P\oplus P$. Consider the parabolic subgroup $W'$ of $W$
generated by $s$ and $t$. The additive closure of $P$ is invariant under the action of the 
$2$-subcategory of $\underline{\cS}$ generated by $\theta_s$ and $\theta_t$.  The decategorification 
of the latter $2$-representation gives a $1$-dimensional $\mathbb{C}[W']$-module on which 
both $\underline{s}$ and $\underline{t}$ act as the scalar $2$. Thus this is the trivial module 
and it is not annihilated by the element $\underline{w'_0}$, where $w'_0$ is the longest element in 
$W'$. This contradicts $w'_0\not\in \mathcal{J}$, see 
Proposition~\ref{propJKL}\eqref{propJKL.5} and the remark after Proposition~\ref{propJKL}.
The assertion of the lemma follows.
\end{proof}

Write $S=\{s_1,s_2,\dots,s_n\}$. Let $\mathbf{M}$ be a transitive $2$-representation of $\underline{\cS}$
with apex $\mathcal{J}$. Choose an ordering 
\begin{equation}\label{eqo1}
P_1,P_2,\dots,P_m 
\end{equation}
on representatives of isomorphism classes
of indecomposable projectives in $\mathbf{M}(\mathtt{i})$ such that, for all $i,j\in\{1,2,\dots,n\}$ such that
$i<j$ and for all $a,b\in\{1,2,\dots,m\}$, the isomorphisms
$\theta_{s_i}\, P_a\cong P_a\oplus P_a$ and $\theta_{s_j}\, P_b\cong P_b\oplus P_b$ imply $a<b$.

\begin{proposition}\label{prop4}
Let $\mathbf{M}$ be a transitive $2$-representation of $\underline{\cS}$
with apex $\mathcal{J}$. Then, with respect to the above ordering, we have
\begin{equation}\label{eqpr4-1}
\mathtt{M}=\left(
\begin{array}{c|c|c|c|c}
2E_1& B_{12} &B_{13}&\dots & B_{1n}\\\hline
B_{21}& 2E_2 &B_{23}&\dots & B_{2n}\\\hline
B_{31}& B_{32} &2E_3 &\dots & B_{3n}\\\hline
\vdots&\vdots&\vdots&\ddots&\vdots\\\hline
B_{n1}& B_{n2} & B_{n3} &\dots & 2E_n
\end{array}
\right),
\end{equation}
where $E_i$, for $1\leq i\leq n$, are non-trivial identity matrices and, 
for $1\leq i\neq j\leq n$, the matrix $B_{ij}$ is non-zero if and only if $s_is_j\neq s_js_i$.
\end{proposition}

\begin{proof}
After Lemmata~\ref{lem2} and \ref{lem3}, it remains to show that 
$B_{ij}$ is non-zero if and only if $s_is_j\neq s_js_i$. If $s_is_j= s_js_i$, then 
$B_{ij}$ is zero as $s_is_j\not\in\mathcal{J}$ and $\mathcal{J}$ is the apex of $\mathbf{M}$.

Assume that $s_is_j\neq s_js_i$ and $B_{ij}=0$. Let as assume that $i>j$, the other case
is similar. Taking into account the previous paragraph and the fact that the underlying
graph of the Coxeter diagram of $W$ is a tree, we may rearrange the order of the elements in 
$S$ such that the matrix $\mathtt{M}$ has the following block form:
\begin{displaymath}
\left(
\begin{array}{c|c}
*&*\\\hline
0&*
\end{array}
\right). 
\end{displaymath}
As any power of such a matrix will contain a non-empty zero south-west corner, we get a contradiction with the
assumption that $\mathbf{M}$ is transitive. The claim follows.
\end{proof}

Let $\{s_1,s_2,\dots,s_n\}$ be an ordering of simple reflections in $S$ such that the block $2E_i$
in \eqref{eqpr4-1} corresponds to $s_i$, for all $i=1,2,\dots,n$, then we have
{\small
\begin{displaymath}
\Lparen\theta_{s_1}\Rparen=\left(
\begin{array}{c|c|c|c|c}
2E_1& B_{12} &B_{13}&\dots & B_{1n}\\\hline
0& 0 &0&\dots & 0\\\hline
0& 0 &0 &\dots & 0\\\hline
\vdots&\vdots&\vdots&\ddots&\vdots\\\hline
0& 0 & 0 &\dots & 0
\end{array}
\right),\quad
\Lparen\theta_{s_2}\Rparen=\left(
\begin{array}{c|c|c|c|c}
0& 0 &0&\dots & 0\\\hline
B_{21}& 2E_2 &B_{23}&\dots & B_{2n}\\\hline
0& 0 &0 &\dots & 0\\\hline
\vdots&\vdots&\vdots&\ddots&\vdots\\\hline
0& 0 & 0 &\dots & 0
\end{array}
\right),
\end{displaymath}
}
and so on.

\subsection{The special $\mathbb{C}[W]$-module for $\mathcal{J}$}\label{s4.3}

Here we describe the special $\mathbb{C}[W]$-module, in the sense of \cite{KM2},
see also Subsection~\ref{s2.7}, associated to $\mathcal{J}$.

\begin{proposition}\label{propspecialJ}
The special $\mathbb{C}[W]$-module associated to $\mathcal{J}$ is 
$V\otimes\mathbb{C}_{\mathrm{sign}}$, where $\mathbb{C}_{\mathrm{sign}}$ is the sign $\mathbb{C}[W]$-module.
\end{proposition}

If $W$ is a Weyl group, this can be derived from \cite{Do}, if one takes into account the difference in 
the normalizations of the Hecke algebra and the choice of the Kazhdan-Lusztig basis. 

\begin{proof}
Let $s$ and $t$ be two different simple reflections in $W$ and $W'$ the corresponding parabolic subgroup
of rank two, with the longest element $w'_0$. The restriction of $V\otimes\mathbb{C}_{\mathrm{sign}}$ to
$W'$ decomposes, by construction, into a direct sum of $V'\otimes\mathbb{C}'_{\mathrm{sign}}$, where 
$V'$ and $\mathbb{C}'_{\mathrm{sign}}$ are the geometric and the sign representations of $W'$, respectively,
and a number of copies of $\mathbb{C}'_{\mathrm{sign}}$. As both $V'\otimes\mathbb{C}'_{\mathrm{sign}}$
and $\mathbb{C}'_{\mathrm{sign}}$ are annihilated by $\underline{w'_0}$, it follows that 
$\underline{w'_0}$ annihilates $V\otimes\mathbb{C}_{\mathrm{sign}}$.

Let $I$ be the ideal in $\mathbb{C}[W]$ with the $\mathbb{C}$-basis $\underline{w}$, where 
$w\not\in\mathcal{J}\cup\{e\}$. From  Proposition~\ref{propJKL}\eqref{propJKL.5}, it follows that 
$I$ is generated by all $\underline{w'_0}$ as in the previous paragraph. Therefore the previous paragraph
implies that the vector space $V\otimes\mathbb{C}_{\mathrm{sign}}$ is, in fact,  a non-zero  $\mathbb{C}[W]/I$-module. 

If $W$ is simply laced, then the cell $\mathbb{C}[W]$-module corresponding to any cell in $\mathcal{J}$
is simple. Hence in this case the claim of the proposition follows directly from the previous paragraph.

Now note that the dimension of $V\otimes\mathbb{C}_{\mathrm{sign}}$ equals the number of left cells in
$\mathcal{J}$. From detailed lists of elements in $\mathcal{J}$ provided in Subsection~\ref{s4.2} we
therefore have that  $V\otimes\mathbb{C}_{\mathrm{sign}}$ is the only simple  $\mathbb{C}[W]/I$-module
of this dimension in types $B_n$ and $F_4$. This implies the claim of the proposition for these types.

For the general dihedral type, one may note that $V\otimes\mathbb{C}_{\mathrm{sign}}\cong V$ and hence
the claim of the proposition follows from \cite{KM2}. 

It remains to consider the two exceptional types $H_3$ and $H_4$. We do explicit computations in both types.
Note that, from the lists in Subsection~\ref{s4.2} we see that there are two non-isomorphic simple 
$\mathbb{C}[W]/I$-modules of the same dimension. 

Let us start with type $H_3$. The action of $\mathbf{s}$ 
(cf. Subsection~\ref{s4.4}) on $C_{\mathcal{L}_1}$ in the basis
\begin{displaymath}
1,121,21,2121,321,32121,  
\end{displaymath}
taken from the corresponding table in  Subsection~\ref{s4.2}, is given by the matrix
\begin{displaymath}
\left(
\begin{array}{cccccc}
2&0&1&0&0&0\\
0&2&1&1&0&0\\
1&1&2&0&1&0\\
0&1&0&2&0&1\\
0&0&1&0&2&0\\
0&0&0&1&0&2 
\end{array}
\right).
\end{displaymath}
The eigenvalue of maximal absolute value for this matrix is $2+\frac{\sqrt{2\sqrt{5}+10}}{2}$.
At the same time, the action of $\mathbf{s}$ on $V\otimes\mathbb{C}_{\mathrm{sign}}$ in the basis
corresponding to $1$, $2$ and $3$, is given by the matrix
\begin{displaymath}
\left(
\begin{array}{ccc}
2&-\frac{1+\sqrt{5}}{2}&0\\
-\frac{1+\sqrt{5}}{2}&2&-1\\
0&-1&2 
\end{array}
\right).
\end{displaymath}
As $2+\frac{\sqrt{2\sqrt{5}+10}}{2}$ is also an eigenvalue for this latter matrix, we obtain our 
statement in type $H_3$.

Type $H_4$ is treated similarly by noticing that the matrices
\begin{displaymath}
\left(
\begin{array}{cccccccc}
2&0&1&0&0&0&0&0\\
0&2&1&1&0&0&0&0\\
1&1&2&0&1&0&0&0\\
0&1&0&2&0&1&0&0\\
0&0&1&0&2&0&1&0\\
0&0&0&1&0&2&0&1\\ 
0&0&0&0&1&0&2&0\\ 
0&0&0&0&0&1&0&2
\end{array}
\right) \qquad\text{ and }\qquad
\left(
\begin{array}{cccc}
2&-\frac{1+\sqrt{5}}{2}&0&0\\
-\frac{1+\sqrt{5}}{2}&2&-1&0\\
0&-1&2&-1\\ 
0&0&-1&2\\ 
\end{array}
\right)
\end{displaymath}
which represent the action of $\mathbf{s}$ on $C_{\mathcal{L}_1}$ and 
$V\otimes\mathbb{C}_{\mathrm{sign}}$, respectively, have a common eigenvalue 
which is of maximal absolute value for both of them
(this eigenvalue is approximately $3.98904$). This completes the proof.
\end{proof}

\begin{remark}\label{remeigen}
{\rm  
In type $H_3$, we used the matrix calculator that can be found at www.mathportal.org.
As the outcome is given as a real number, it is easily checked by hand.
In type $H_4$ we used the matrix calculator available at
www.arndt-bruenner.de and cross-checked the result on the matrix calculator available at 
www.bluebit.gr.
}
\end{remark}

\subsection{$1$-morphisms act as projective functors}\label{s5.5}

Our aim in this subsection is to prove the following very general result which extends
and unifies  \cite[Lemma~12]{MM5}, \cite[Lemma~6.14]{Zi} and \cite[Lemma~10]{MaMa}.

\begin{theorem}\label{thmminimal}
Let $\mathbf{M}$ be a simple transitive $2$-representation  of a fiat $2$-category $\cC$.
Let, further,  $\mathcal{I}$ be the apex of $\mathbf{M}$ and $\mathrm{F}\in \mathcal{I}$.
\begin{enumerate}[$($i$)$]
\item\label{thmminimal.1} For every object $X$ in any $\overline{\mathbf{M}}(\mathtt{i})$, 
the object $\mathrm{F}\,X$ is projective. 
\item\label{thmminimal.2} The functor $\overline{\mathbf{M}}(\mathrm{F})$ is a projective functor.
\end{enumerate}
\end{theorem}

\begin{proof}
Without loss of generality we may assume that $\mathcal{I}$ is the maximum
two-sided cell of $\cC$.

Denote by $Q$ the complexification of the split Grothendieck group of
\begin{displaymath}
\coprod_{\mathtt{i}\in\ccC}{\mathbf{M}}(\mathtt{i}). 
\end{displaymath}
Let $\mathbf{B}$ denote the distinguished basis in $Q$ given by classes of indecomposable modules.
Let $\mathrm{F}_1,\mathrm{F}_2,\dots,\mathrm{F}_k$ be a 
complete and irredundant list of indecomposable $1$-morphisms in $\mathcal{I}$. For $i=1,2,\dots,k$, 
let $\mathcal{X}^{(i)}_{\bullet}$ be a minimal projective resolution of $\mathrm{F}_i\, X$. Further,
for $j\geq 0$, denote by $v_j^{(i)}$ the image of $\mathcal{X}^{(i)}_{j}$ in $Q$.
Note that  $v_j^{(i)}$ is a linear combination of elements in $\mathbf{B}$
with non-negative integer coefficients.

As $\cC$ is assumed to be fiat, there exist $i,j\in\{1,2,\dots,k\}$ such that 
$\mathrm{F}_i\circ \mathrm{F}_j\neq 0$. Indeed, one can, for example, take any $i$ and let 
$\mathrm{F}_j$ be the Duflo involution in the left cell of $\mathrm{F}_i$, see \cite[Proposition~17]{MM1}.
Therefore we may apply \cite[Proposition~18]{KM2}, which asserts that the algebra $A_{\ccC}$ 
has a unique idempotent $e$ of the form 
\begin{displaymath}
\sum_{i=1}^kc_i[\mathrm{F}_i], 
\end{displaymath}
where all $c_i\in\mathbb{R}_{>0}$ and $[\mathrm{F}_i]$ denotes the image of $\mathrm{F}_i$ in $A_{\ccC}$.
The vector space $Q$ is, naturally, an $A_{\ccC}$-module. 

Let $j\geq 0$ and set
\begin{displaymath}
v(j):=\sum_{i=1}^kc_iv_j^{(i)}.
\end{displaymath}
Applying any $1$-morphism $\mathrm{G}$ to a minimal projective resolution of some object $Y$
and taking into account that the action of $\mathrm{G}$ is exact as $\cC$ is fiat, gives a
projective resolution of $\mathrm{G}\,Y$ which, a priori, does not have to be minimal.
By construction and minimality of $\mathcal{X}^{(i)}_{\bullet}$, this implies that  
$ev(j)-v(j)$ is a  linear combination of elements in $\mathbf{B}$ with non-negative 
real coefficients. Note that transitivity of $\mathbf{M}$ implies that, for any linear combination
$z$ of elements in $\mathbf{B}$ with non-negative real coefficients, the element $ez$ is also
a linear combination of elements in $\mathbf{B}$ with non-negative real coefficients, moreover,
if $z\neq 0$, then $ez\neq 0$.
Therefore the equality $e^2=e$ yields $ev(j)=v(j)$. Consequently, 
the projective resolution $\mathrm{F}_i\mathcal{X}^{(l)}_{\bullet}$ of $\mathrm{F}_i(\mathrm{F}_l\, X)$
is minimal, for all $i$ and $l$. 

Now the proof is completed by standard arguments as in \cite[Lemma~12]{MM5}. From the previous paragraph
it follows that homomorphisms $\mathcal{X}^{(l)}_{1}\to \mathcal{X}^{(l)}_{0}$,
for $l=1,2,\dots,k$, generate a $\cC$-invariant  ideal of $\mathbf{M}$ different from $\mathbf{M}$. 
Because of simple transitivity of $\mathbf{M}$, this ideal must be zero.
Therefore all homomorphisms $\mathcal{X}^{(l)}_{1}\to \mathcal{X}^{(l)}_{0}$ are zero which implies
$\mathcal{X}^{(l)}_{0}\cong \mathrm{F}_l\, X$. Claim~\eqref{thmminimal.1} follows.
Claim~\eqref{thmminimal.2} follows from claim~\eqref{thmminimal.1} and \cite[Lemma~13]{MM5}.
\end{proof}

\subsection{The matrix $\mathtt{M}$ is symmetric}\label{s4.37}

Let $\mathbf{M}$ be a simple transitive $2$-representation of $\underline{\cS}$ with apex $\mathcal{J}$.
Consider $\overline{\mathbf{M}}$ and let $P_1$, $P_2$,\dots, $P_k$ be a full list of pairwise non-isomorphic
indecomposable projectives in $\overline{\mathbf{M}}(\mathtt{i})$. Let $L_1$, $L_2$,\dots, $L_k$ be their
respective simple tops. 

\begin{proposition}\label{propn31}
For every $i\in\{1,2,\dots,k\}$, we have $\mathrm{F}_{\mathrm{pr}}\, L_i\cong P_i$.
\end{proposition}

\begin{proof}
We know from Theorem~\ref{thmminimal}\eqref{thmminimal.1} that  
$\mathrm{F}_{\mathrm{pr}}\, L_i$ is projective. From the matrix $\mathtt{M}$ described in 
Subsection~\ref{s4.4}, we see that there is a unique $s\in S$ such that $\theta_s\, L_i\neq 0$.
Therefore $\mathrm{F}_{\mathrm{pr}}\, L_i\cong \theta_s\, L_i$.
From the matrix $\llbracket \theta_s\rrbracket$, 
see Subsections~\ref{s2.6} and \ref{s4.4} in combination with the remark that 
each $\theta_s$ is self-adjoint, we see that $\theta_s\,L_i$ has two subquotients
isomorphic to $L_i$ and all other subquotients are killed by $\theta_s$. Note that $\theta_s\,L_i$
has length at least three, as each row of $\mathtt{M}$ must contain at least one non-zero block $B_{ij}$
by Proposition~\ref{prop4}. This means that $\theta_s\,L_i$ does contain some subquotients which are not isomorphic 
to $L_i$. If $L_j$ is such a subquotient, then, by adjunction,
\begin{displaymath}
\mathrm{Hom}_{\overline{\mathbf{M}}(\mathtt{i})}(\theta_s\,L_i,L_j)=
\mathrm{Hom}_{\overline{\mathbf{M}}(\mathtt{i})}(L_i,\theta_s\,L_j)
=\mathrm{Hom}_{\overline{\mathbf{M}}(\mathtt{i})}(L_i,0)=0
\end{displaymath}
and $L_j$ does not appear in the top of $\theta_s\,L_i$. Similarly, $L_j$ does not appear in the socle
of  $\theta_s\,L_i$ either. It follows that $\theta_s\,L_i$ can only have $L_i$ in top and socle.
This means that one of the $L_i$'s is in the top of $\theta_s\,L_i$ and the other one is in the socle.
In particular, $\theta_s\,L_i$ is indecomposable. Therefore
$\theta_s\,L_i\cong P_i$ by the projectivity of $\theta_s\,L_i$. This completes the proof.
\end{proof}

\begin{corollary}\label{corn32}
Let $\mathbf{M}$ be a simple transitive $2$-representation of $\underline{\cS}$ with apex $\mathcal{J}$.
Then the matrix $\mathtt{M}$ is symmetric.
\end{corollary}

\begin{proof}
By Proposition~\ref{propn31}, the row $i$ of $\mathtt{M}$ describes composition multiplicities of
simple modules in $P_i$. By definition, the column $i$ of $\mathtt{M}$ describes the direct sum decomposition
of $\mathrm{F}_{\mathrm{pr}}\,P_i$. Note that $\mathrm{F}_{\mathrm{pr}}$ is exact as $\cS$ is fiat.
Therefore, from Proposition~\ref{propn31} we have that the multiplicity of $P_j$ as a direct summand in 
$\mathrm{F}_{\mathrm{pr}}\,P_i$ coincides with the composition multiplicity of $L_j$ in $P_i$.
The claim follows.
\end{proof}

\section{Some spectral properties of integral matrices}\label{sim}

\subsection{Setup}\label{sim.1}

Denote by $\mathbb{M}$ the set of all matrices with non-negative integer coefficients, that is
\begin{displaymath}
\mathbb{M}:=\bigcup_{k,m\in\mathbb{Z}_{>0}}\mathrm{Mat}_{k\times m}(\mathbb{Z}_{\geq 0}).
\end{displaymath}
For each $X\in \mathbb{M}$, we can consider the symmetric matrices $XX^{\mathrm{tr}}$ and 
$X^{\mathrm{tr}}X$. Each of these two matrices is diagonalizable over $\mathbb{R}$, moreover,
all eigenvalues are non-negative. Furthermore, these two matrices 
have the same spectrum (considered as a multiset)
with only possible exception of the multiplicity of the eigenvalue zero. We denote by 
$\mathbf{m}_X$ the eigenvalue of $XX^{\mathrm{tr}}$ or, equivalently, $X^{\mathrm{tr}}X$
with the maximal absolute value. Note that $\mathbf{m}_X=0$ if and only if $X$ is the zero matrix.
The aim of this section is to describe all $X\in \mathbb{M}$ such that $\mathbf{m}_X<4$.

\subsection{Staircase matrices}\label{sim.2}

A {\em staircase matrix} is a matrix of the form
\begin{equation}\label{simeq1}
\left(
\begin{array}{cccccccc}
1&1&0&0&0&\dots&0&0 \\
0&1&1&0&0&\dots&0&0  \\
0&0&1&1&0&\dots&0&0 \\ 
0&0&0&1&1&\dots&0&0 \\ 
0&0&0&0&1&\dots&0&0 \\ 
\vdots&\vdots&\vdots&\vdots&\vdots&\ddots&\vdots&\vdots \\
0&0&0&0&0&\dots&1&1\\ 
0&0&0&0&0&\dots&0&1 
\end{array}
\right) \quad
\left(
\begin{array}{ccccccccc}
1&1&0&0&0&\dots&0&0&0  \\
0&1&1&0&0&\dots&0&0&0   \\
0&0&1&1&0&\dots&0&0&0  \\ 
0&0&0&1&1&\dots&0&0&0  \\ 
0&0&0&0&1&\dots&0&0&0  \\ 
\vdots&\vdots&\vdots&\vdots&\vdots&\ddots&\vdots&\vdots&\vdots \\
0&0&0&0&0&\dots&1&1&0\\
0&0&0&0&0&\dots&0&1&1 
\end{array}
\right) 
\end{equation}
or its transpose. Note that each staircase matrix is of the
size $k\times k$ or $k\times (k+1)$ or $(k+1)\times k$, for some positive integer $k$.

An {\em extended staircase matrix} is a matrix obtained from a staircase matrix by adding
one additional row (or column, but not both) as shown below. In the pictures,
the line separates the original staircase matrix from the added row (respectively, column).
The picture below shows extended staircase matrices for the left matrix
in \eqref{simeq1}.
\begin{displaymath}
\left(
\begin{array}{c|cccccccc}
1&1&1&0&0&0&\dots&0&0 \\
0&0&1&1&0&0&\dots&0&0  \\
0&0&0&1&1&0&\dots&0&0 \\ 
0&0&0&0&1&1&\dots&0&0 \\ 
0&0&0&0&0&1&\dots&0&0 \\ 
\vdots&\vdots&\vdots&\vdots&\vdots&\vdots&\ddots&\vdots&\vdots \\
0&0&0&0&0&0&\dots&1&1\\ 
0&0&0&0&0&0&\dots&0&1 
\end{array}
\right) \quad
\left(
\begin{array}{cccccccc}
1&1&0&0&0&\dots&0&0 \\
0&1&1&0&0&\dots&0&0  \\
0&0&1&1&0&\dots&0&0 \\ 
0&0&0&1&1&\dots&0&0 \\ 
0&0&0&0&1&\dots&0&0 \\ 
\vdots&\vdots&\vdots&\vdots&\vdots&\ddots&\vdots&\vdots \\
0&0&0&0&0&\dots&1&1\\ 
0&0&0&0&0&\dots&0&1 \\ \hline
0&0&0&0&0&\dots&0&1
\end{array}
\right) 
\end{displaymath}
The next picture shows extended staircase matrices for the right matrix
in \eqref{simeq1}.
\begin{displaymath}
\left(
\begin{array}{c|ccccccccc}
1&1&1&0&0&0&\dots&0&0&0  \\
0&0&1&1&0&0&\dots&0&0&0   \\
0&0&0&1&1&0&\dots&0&0&0  \\ 
0&0&0&0&1&1&\dots&0&0&0  \\ 
0&0&0&0&0&1&\dots&0&0&0  \\ 
\vdots&\vdots&\vdots&\vdots&\vdots&\vdots&\ddots&\vdots&\vdots&\vdots \\
0&0&0&0&0&0&\dots&1&1&0\\
0&0&0&0&0&0&\dots&0&1&1 
\end{array}
\right)\quad
\left(
\begin{array}{ccccccccc|c}
1&1&0&0&0&\dots&0&0&0&0   \\
0&1&1&0&0&\dots&0&0&0&0    \\
0&0&1&1&0&\dots&0&0&0&0   \\ 
0&0&0&1&1&\dots&0&0&0&0   \\ 
0&0&0&0&1&\dots&0&0&0&0   \\ 
\vdots&\vdots&\vdots&\vdots&\vdots&\ddots&\vdots&\vdots&\vdots&\vdots \\
0&0&0&0&0&\dots&1&1&0&0 \\
0&0&0&0&0&\dots&0&1&1&1 
\end{array}
\right)
\end{displaymath}
For the transposes of \eqref{simeq1}, the above pictures should also be transposed.
Note that each extended staircase matrix is of the
size $k\times (k+1)$ or $k\times (k+2)$ or $(k+1)\times k$ or $(k+2)\times k$, 
for some positive integer $k$.

\subsection{The main result}\label{sim.3}

\begin{proposition}\label{propsim}
Let $X\in\mathbb{M}$ be such that both $XX^{\mathrm{tr}}$ and $X^{\mathrm{tr}}X$
are irreducible and  $\mathbf{m}_X<4$. Then, using independent permutations
of rows and columns, $X$ can be reduced to a staircase matrix or an extended staircase 
matrix or one of the following matrices (or their transposes):
\begin{displaymath}
X_1:=\left(
\begin{array}{ccc}
1&0&0\\
1&1&1\\
0&0&1
\end{array}
\right),\quad 
X_2:=\left(
\begin{array}{cccc}
1&1&0&0\\
0&1&1&1\\
0&0&0&1
\end{array}
\right),\quad 
X_3:=\left(
\begin{array}{cccc}
1&0&0&0\\
1&1&0&0\\
0&1&1&1\\
0&0&0&1
\end{array}
\right). 
\end{displaymath}
\end{proposition}

The statement of Proposition~\ref{propsim} can be interpreted in the way that there is a
natural correspondence between matrices $X$ as in the proposition and simply laced Dynkin diagrams.
Staircase matrices correspond to type $A$, extended staircase matrices correspond to type $D$ and
the three exceptional matrices correspond to type $E$.
Starting from a matrix $X$ appearing in the classification provided 
by Proposition~\ref{propsim}, we can consider the matrix
\begin{displaymath}
\left(\begin{array}{cc}2E&X\\X^{\mathrm{tr}}&2E\end{array}\right) 
\end{displaymath}
and pretend that it appears as $\mathtt{M}$ in some $2$-representation.
If we compute the underlying algebra of that $2$-representation
as in many examples later on, see e.g. Subsection~\ref{s6.9},
we will get a doubling of a simply laced  Dynkin quiver.

\subsection{Proof of Proposition~\ref{propsim}}\label{sim.4}

Let $X\in\mathbb{M}$ be such that both $XX^{\mathrm{tr}}$ and $X^{\mathrm{tr}}X$
are irreducible and $\mathbf{m}_X<4$.

Assume that $X$ has an entry that is greater than or equal to $2$. Then 
$XX^{\mathrm{tr}}$ has a diagonal entry which is greater than or equal to $4$.
Let $\lambda$ be the Perron-Frobenius eigenvalue of $XX^{\mathrm{tr}}$.
Then, by the Perron-Frobenius Theorem, the limit
\begin{displaymath}
\lim_{i\to\infty}\frac{(XX^{\mathrm{tr}})^i}{\lambda^i} 
\end{displaymath}
exists, which means that $\lambda\geq 4$, a contradiction. Therefore $X$ is a $0$-$1$-matrix.

For a matrix $M$, a submatrix $N$ of $M$ is the matrix obtained by taking entries in the intersection
of a non-empty set of rows of $M$ and a non-empty set of columns of $M$. If $N$ is a submatrix of $M$,
then, clearly, $\mathbf{m}_N\leq \mathbf{m}_M$. An argument similar to the 
one in the previous paragraph shows that $X$ cannot have any submatrix which is equal to either
of the following matrices, nor their transposes,
\begin{equation}\label{simeq2}
\left(\begin{array}{cc}1&1\\1&1\end{array}\right),\qquad 
\left(\begin{array}{cccc}1&1&1&1\end{array}\right). 
\end{equation}
Taking the first matrix in \eqref{simeq2} into account and using the fact that both 
$XX^{\mathrm{tr}}$ and $X^{\mathrm{tr}}X$ are irreducible, we see that, by independent permutation of row
and columns, $X$ can be reduced to the form
\begin{displaymath}
\left(\begin{array}{cccccccccc}
1&\dots&1&1&0&\dots&0&0&0&\dots\\
0&\dots&0&1&1&\dots&1&1&0&\dots\\
0&\dots&0&0&0&\dots&0&1&1&\dots\\
\vdots&\vdots&\vdots&\vdots&\vdots&\vdots&\vdots&\vdots&\vdots&\vdots
\end{array}
\right). 
\end{displaymath}
So, from now on we may assume that $X$ is of the latter form.
Taking the second matrix in \eqref{simeq2} into account, we see that each row and each
column of $X$ contains at most three non-zero entries.

Next we claim that the total number of rows and columns in $X$ which contain three non-zero 
entries is at most one. For this we have to exclude, up to transposition, two types of 
possible submatrices in $X$. The first one is the matrix
\begin{displaymath}
\left(\begin{array}{cccccccccc}
1&0&0&0&\dots&\dots&0&0&0&0\\
1&0&0&0&\dots&\dots&0&0&0&0\\
1&1&0&0&\dots&\dots&0&0&0&0\\
0&1&1&0&\dots&\dots&0&0&0&0\\
\vdots&\vdots&\vdots&\vdots&\ddots&\ddots&\vdots&\vdots&\vdots&\vdots\\
\vdots&\vdots&\vdots&\vdots&\ddots&\ddots&\vdots&\vdots&\vdots&\vdots\\
0&0&0&0&\dots&\dots&1&1&0&0\\
0&0&0&0&\dots&\dots&0&1&1&1
\end{array}
\right). 
\end{displaymath}
For this matrix it is easy to check that the vector $(2,2,\dots,2,1,1)^{\mathrm{tr}}$ is an eigenvector
of $X^{\mathrm{tr}}X$ with eigenvalue $4$, leading to a contradiction. The second one is the matrix
\begin{displaymath}
\left(\begin{array}{cccccccccc}
1&1&1&0&0&\dots&0&0&0&0\\
0&0&1&1&0&\dots&0&0&0&0\\
0&0&0&1&1&\dots&0&0&0&0\\
\vdots&\vdots&\vdots&\vdots&\ddots&\ddots&\vdots&\vdots&\vdots&\vdots\\
\vdots&\vdots&\vdots&\vdots&\ddots&\ddots&\vdots&\vdots&\vdots&\vdots\\
0&0&0&0&0&\dots&1&1&0&0\\
0&0&0&0&0&\dots&0&1&1&1
\end{array}
\right). 
\end{displaymath}
For this matrix it is easy to check that the vector $(1,1,\dots,1)^{\mathrm{tr}}$ is an eigenvector
of $XX^{\mathrm{tr}}$ with eigenvalue $4$, leading again to a contradiction.

If $X$ has neither rows nor columns with three non-zero entries, then $X$ is a staircase matrix.
By the above, if $X$ has a row or a column with three non-zero entries, it is unique. Now, the claim of
the proposition follows from the observation that, for the following matrices $X$:
\begin{displaymath}
\left(\begin{array}{cccc}1&1&0&0\\0&1&0&0\\0&1&1&0\\0&0&1&1\\0&0&0&1\\\end{array}\right),\qquad 
\left(\begin{array}{ccc}1&0&0\\1&1&0\\0&1&0\\0&1&1\\0&0&1\\\end{array}\right), 
\end{displaymath}
the matrix $X^{\mathrm{tr}}X$ has $4$ as an eigenvalue.

\section{Simple transitive $2$-representations of $\cS$ in Coxeter type $I_2(n)$ with $n$ odd}\label{s17}

\subsection{Setup and the main result}\label{s17.1}

In this section we assume that $W$ is of Coxeter type $I_2(n)$, for some $n\geq 3$, and 
$S=\{s,t\}$ with the Coxeter diagram
\begin{displaymath}
\xymatrix{s\ar@{-}[r]^n&t.} 
\end{displaymath}
In this case there are three two-sided cells, namely 
\begin{itemize}
\item the two-sided cell of the identity element;
\item the two-sided cell of $w_0$;
\item the two-sided cell $\mathcal{J}$.
\end{itemize}
Our aim in this section is to prove the following.

\begin{theorem}\label{thm1700}
For $n\geq 3$ odd, every simple transitive $2$-representation of 
$\cS$ is a cell $2$-representation.
\end{theorem}

If $\mathbf{M}$ is a simple transitive $2$-representation of $\cS$ whose apex is not $\mathcal{J}$, then 
$\mathbf{M}$ is a cell $2$-representation by the same argument as in \cite[Theorem~18]{MM5},
see also \cite{MaMa,Zi} for similar arguments. Taking \cite[Theorem~19]{MM2} into account, 
from now on, we may assume that $\mathbf{M}$ is a simple transitive $2$-representation of $\underline{\cS}$ 
with apex $\mathcal{J}$. Our goal is to prove that $\mathbf{M}$ is a cell $2$-representation.

\subsection{Fibonacci polynomials in disguise}\label{s17.2}

For $i=0,1,2,\dots$, we define, recursively, polynomials $f_i(x)\in\mathbb{Z}[x]$ as follows:
$f_0(x)=0$, $f_1(x)=1$, and, for $i>1$, 
\begin{displaymath}
f_i(x)=
\begin{cases}
f_{i-1}(x)-f_{i-2}(x),& i\text{ is odd};\\
xf_{i-1}(x)-f_{i-2}(x),& i\text{ is even}.
\end{cases}
\end{displaymath}
Coefficients of $f_i$ are given by Sequence~A115139 in \cite{OEIS}. 
The values of $f_i$ for small $i$ are given here:
\begin{displaymath}
\begin{array}{c||l|l}
i&f_i(x)&\text{ factorization over }\mathbb{Q}\\
\hline\hline
0&0\\
1&1&1\\
2&x&x\\
3&x-1&x-1\\
4&x^2-2x&x(x-2)\\
5&x^2-3x+1&x^2-3x+1\\
6&x^3-4x^2+3x&x(x-1)(x-3)\\
7&x^3-5x^2+6x-1&x^3-5x^2+6x-1\\
8&x^4-6x^3+10x^2-4x&x(x-2)(x^2-4x+2)\\
9&x^4-7x^3+15x^2-10x+1&(x-1)(x^3-6x^2+9x-1)\\
10&x^5-8x^4+21x^3-20x^2+5x&x(x^2-3x+1)(x^2-5x+5)\\
11&x^5-9x^4+28x^3-35x^2+15x-1&x^5-9x^4+28x^3-35x^2+15x-1\\
12&x^6-10x^5+36x^4-56x^3+35x^2-6x&x(x-1)(x-2)(x-3)(x^2-4x+1).
\end{array} 
\end{displaymath}

For $i=0,1,2,\dots$, we define, recursively, polynomials $g_i(x)\in\mathbb{Z}[x]$ in the following way:
$g_0(x)=0$, $g_1(x)=1$, and $g_i(x)=xg_{i-1}(x)+g_{i-2}(x)$, for $i>1$. These
polynomials are called {\em Fibonacci polynomials}, see Sequence~A011973 in \cite{OEIS}.
By comparing the two definitions, we have,  
\begin{displaymath}
\begin{array}{rclr}
(-1)^{\lfloor\frac{i}{2}\rfloor}f_i(-x^2)&=&g_i(x),&i\text{ is odd};\\
\frac{(-1)^{\frac{i}{2}}}{x}f_i(-x^2)&=&g_i(x),&i\text{ is even}.
\end{array}
\end{displaymath}
Now, from \cite[Lemma~5]{Lev}, it follows that, for each $i\in\{1,2,3,\dots\}$, the polynomial 
$f_i(x)$ has a unique irreducible (over $\mathbb{Q}$) factor, denoted $\underline{f}_i(x)$, which
is not an irreducible factor of any $f_j(x)$, for $j<i$. Furthermore,
\begin{displaymath}
f_i(x)=\prod_{d\vert i}\underline{f}_d(x).
\end{displaymath}
From \cite[Definition~1]{Lev}, it follows that, for $i>2$, the polynomial  $\underline{f}_i(x)$ 
has degree $\frac{\phi(i)}{2}$, where $\phi$ is Euler's totient function, and all roots of 
$\underline{f}_i(x)$ are positive real numbers less than $4$. Furthermore, 
the sequence of maximal roots of $\underline{f}_i(x)$ is strictly increasing and converges to 
$4$, when $i\to\infty$. For small $i$, the polynomials $\underline{f}_i(x)$ are given in the following table:
\begin{displaymath}
\begin{array}{c||l}
i&\underline{f}_i(x)\\
\hline\hline
0&0\\
1&1\\
2&x\\
3&x-1\\
4&x-2\\
5&x^2-3x+1\\
6&x-3\\
7&x^3-5x^2+6x-1\\
8&x^2-4x+2\\
9&x^3-6x^2+9x-1\\
10&x^2-5x+5\\
11&x^5-9x^4+28x^3-35x^2+15x-1\\
12&x^2-4x+1\\
13&x^6-11x^5+45x^4-84x^3+70x^2-21x+1\\
14&x^3-7x^2+14x-7\\
15&x^4-9x^3+26x^2-24x+1.
\end{array} 
\end{displaymath}
We refer the reader to \cite{Lev,WP} for further properties of the above polynomials.

\subsection{Disguised Fibonacci polynomials and $2$-representations of $\underline{\cS}$}\label{s17.3}

Recall that we are in Coxeter type $I_2(n)$, for some $n\geq 3$. Let $\mathbf{M}$ be a simple 
transitive $2$-representation of $\underline{\cS}$ with apex $\mathcal{J}$. By 
Subsection~\ref{s4.4} and Corollary~\ref{corn32}, we have
\begin{displaymath}
\mathtt{M}=\left( 
\begin{array}{c|c}
2E_1&B\\\hline
B^{\mathrm{tr}}&2E_2
\end{array}
\right),\quad
\Lparen\theta_s\Rparen=\left( 
\begin{array}{c|c}
2E_1&B\\\hline
0&0
\end{array}
\right),\quad
\Lparen\theta_t\Rparen=\left( 
\begin{array}{c|c}
0&0\\\hline
B^{\mathrm{tr}}&2E_2
\end{array}
\right).
\end{displaymath}

\begin{lemma}\label{lem1701}
Both $BB^{\mathrm{tr}}$ and $B^{\mathrm{tr}}B$ are annihilated by $f_n(x)$. 
\end{lemma}

\begin{proof}
Set $C=B^{\mathrm{tr}}$.
For simplicity, for $i\geq 1$, let $s_i=stst\dots$ be the element of length $i$ and 
$t_i=tsts\dots$ be the element of length $i$. Using induction on $i$ and the multiplication rule
\begin{equation}\label{eq1705}
\theta_s\theta_{t_i}\cong
\begin{cases}
\theta_{s_2}, & i=1;\\
\theta_{s_{i+1}}\oplus \theta_{s_{i-1}}, & i>1;
\end{cases}
\qquad
\theta_t\theta_{s_i}\cong
\begin{cases}
\theta_{t_2}, & i=1;\\
\theta_{t_{i+1}}\oplus \theta_{t_{i-1}}, & i>1;
\end{cases}
\end{equation}
one proves, by induction, that 
\begin{displaymath}
\Lparen\theta_{s_{i}}\Rparen=\left( 
\begin{array}{c|c}
2f_i(BC)&f_i(BC)B\\\hline
0&0
\end{array}
\right),\quad
\Lparen\theta_{t_{i}}\Rparen=\left( 
\begin{array}{c|c}
0&0\\\hline
f_i(CB)C&2f_i(CB)
\end{array}
\right),\quad
\end{displaymath}
if $i$ is odd, and
\begin{displaymath}
\Lparen\theta_{s_{i}}\Rparen=\left( 
\begin{array}{c|c}
f_i(BC)&2f_i(BC)C^{-1}\\\hline
0&0
\end{array}
\right),\quad
\Lparen\theta_{t_{i}}\Rparen=\left( 
\begin{array}{c|c}
0&0\\\hline
2f_i(CB)B^{-1}&f_i(CB)
\end{array}
\right),\quad
\end{displaymath}
if $i$ is even (here $C^{-1}$ just means that the rightmost appearances of  $C$ in $f_i(BC)$
should be deleted, and similarly for $B^{-1}$ with respect to $f_i(CB)$). Note that
$C (BC)^l = (CB)^l C$, for any $l\in\{0,1,2,\dots\}$, and therefore  $B f_i(CB) C = BC f_i(BC)$,
for any $i$.

As $\mathbf{l}(w_0)=n$, 
the claim now follows from our assumption that 
$\mathbf{M}$ has apex $\mathcal{J}$  and therefore $\Lparen\theta_{w_0}\Rparen=0$. 
\end{proof}

\begin{corollary}\label{cor1702}
If $n$ is odd, then $B$ is invertible.
\end{corollary}

\begin{proof}
If $n$ is odd, then $f_n(x)$ has a non-zero constant term. Therefore, from Lemma~\ref{lem1701} we have that
both $BB^{\mathrm{tr}}$ and $B^{\mathrm{tr}}B$ are invertible and thus $B$ is invertible as well.
\end{proof}

\begin{corollary}\label{cor1705}
Let $n\in\{3,4,5,\dots\}$ and $\mathbf{M}$ be a simple transitive $2$-representation of 
$\underline{\cS}$ in Coxeter type $I_2(n)$ with apex $\mathcal{J}$ and with the corresponding matrix $\mathtt{M}$.
Let $n'\in\{3,4,5,\dots\}$ and $\mathbf{M}'$ be a simple transitive $2$-representation of 
$\underline{\cS}$ in Coxeter type $I_2(n')$ with apex $\mathcal{J}$ and with the corresponding matrix $\mathtt{M}'$.
If $\mathtt{M}=\mathtt{M}'$, then $n=n'$.
\end{corollary}

\begin{proof}
Let $w$ denote the word $stst\dots$ of length $n$. Then $w=w_0$ in Coxeter type $I_2(n)$.
From Lemma~\ref{lem1701}, we have that $\mathbf{M}'(\theta_w)=0$. Therefore $w\not\in\mathcal{J}$ 
in Coxeter type $I_2(n')$ and we obtain $n'\leq n$. By symmetry, we also have $n\leq n'$ and hence $n=n'$.
\end{proof}

\subsection{The explicit form of $B$}\label{s17.4}

\begin{proposition}\label{prop1704}
If $n=2k+1$, then, by independent permutations of rows and columns, 
the matrix $B$ can be reduced to a $k\times k$ staircase matrix.
\end{proposition}

\begin{proof}
If we can prove that, by independent permutations of rows and columns, $B$ 
can be reduced to a square staircase matrix, then the fact that the size of such
matrix is $k\times k$ follows from Corollary~\ref{cor1705}.
From Corollary~\ref{cor1702}, we have that $B$ is an invertible matrix,
say of size $m\times m$. Set $Q=B^{\mathrm{tr}}B$. Then, from Lemma~\ref{lem1701}
and Subsection~\ref{s17.2}, we have that all eigenvalues of $Q$ are positive 
real numbers that are strictly less than $4$.

Therefore, by Proposition~\ref{propsim}, the matrix $B$ is either a staircase matrix
or coincides with one of the matrices $X_1$ or $X_3$. 
If $B=X_3$, then 
\begin{displaymath}
Q=\left(\begin{array}{cccc}2&1&0&0\\1&2&1&1\\0&1&1&1\\0&1&1&2\end{array}\right).
\end{displaymath}
The characteristic polynomial of $Q$ is $x^4-7x^3+14x^2-8x+1$ and is irreducible over $\mathbb{Q}$.
The only odd $i$, for which $\phi(i)=8$, is $i=15$. However, we already know that 
$\underline{f}_{15}(x)=x^4-9x^3+26x^2-24x+1$. Therefore $\underline{f}_{i}(Q)\neq 0$
for any odd $i$. This means that such $B$ is not possible. 

If $B=X_1$, we have
\begin{displaymath}
Q=\left(\begin{array}{ccc}2&1&1\\1&1&1\\1&1&2\end{array}\right).
\end{displaymath}
The characteristic polynomial of $Q$ has factorization $(x-1)(x^2-4x+1)$.
This means that $\underline{f}_{12}(x)$ is a factor of any annihilating polynomial of 
$Q$ and thus  $\underline{f}_{i}(Q)\neq 0$ for any odd $i$.  
The claim of the proposition follows.
\end{proof}

\subsection{Proof of Theorem~\ref{thm1700}}\label{s17.6}

We are now ready to prove Theorem~\ref{thm1700}. Let $\mathbf{M}$ be a simple transitive 
$2$-representation of $\underline{\cS}$ with apex $\mathcal{J}$. Proposition~\ref{prop1704}
gives explicitly the matrix $\mathtt{M}$, in particular, it shows that this matrix is
uniquely determined. Hence this matrix coincides with the corresponding matrix for the
cell $2$-representation $\mathbf{C}_{\mathcal{L}_s}$. 

Consider the abelianization $\overline{\mathbf{M}}$.
Now, starting from the simple module $L_{k}$, we can use Proposition~\ref{propn31} to get
$\theta_s\, L_k\cong P_k$. Using the explicit formulae for the actions of $\theta_s$ and
$\theta_t$ together with \eqref{eq1705}, we see that, applying $\theta_w$, for 
$w\in \mathcal{L}_s$, to $L_k$, we obtain all indecomposable projective objects in 
$\overline{\mathbf{M}}(\mathtt{i})$.

The rest of the proof is now similar to the corresponding parts of the proofs in the literature,
see \cite[Proposition~9]{MM5}, \cite[Sections~6~and~9]{MZ}  or \cite[Subsection~4.9]{MaMa}.
Let $\mathbf{N}$ be the additive $2$-representation of $\cS$ obtained by restricting the action of 
$\cS$ to the category of projective objects in $\overline{\mathbf{M}}(\mathtt{i})$. Then
$\mathbf{N}$ is equivalent to $\mathbf{M}$, see \cite[Theorem~11]{MM2}.
There is a unique strong $2$-natural transformation from $\mathbf{P}_{\mathtt{i}}$ to 
$\overline{\mathbf{M}}$ sending $\mathbbm{1}_{\mathtt{i}}$ to $L_k$. It induces 
a strong $2$-natural transformation from $\mathbf{C}_{\mathcal{L}_s}$ to $\mathbf{N}$.
Since $\mathbf{C}_{\mathcal{L}_s}(\mathtt{i})$ is simple transitive with apex $\mathcal{J}$,
everything that we established earlier for $\mathbf{M}$ (in particular, about the structure
of projective modules etc.) also holds for $\mathbf{C}_{\mathcal{L}_s}$. In particular, 
the Cartan matrices of the underlying algebras of $\mathbf{C}_{\mathcal{L}_s}(\mathtt{i})$
and $\mathbf{N}(\mathtt{i})$ agree. Therefore the above $2$-natural transformation 
is an equivalence between these two categories. This completes the proof.

\section{Simple transitive $2$-representations of $\cS$ in Coxeter type $I_2(n)$ with $n$ even}\label{s6}

\subsection{Setup and preliminaries}\label{s6.1}

In this section we assume that $W$ is of type $I_2(n)$ with $n=2k>4$ and $S=\{s,t\}$ with the Coxeter diagram
\begin{displaymath}
\xymatrix{s\ar@{-}[r]^{n}&t.} 
\end{displaymath}
Let $\mathbf{M}$ be a simple transitive $2$-representation of $\underline{\cS}$ with apex $\mathcal{J}$.
By Subsection~\ref{s4.4} and Corollary~\ref{corn32}, we have
\begin{displaymath}
\mathtt{M}=\left( 
\begin{array}{c|c}
2E_1&B\\\hline
B^{\mathrm{tr}}&2E_2
\end{array}
\right),\quad
\Lparen\theta_s\Rparen=\left( 
\begin{array}{c|c}
2E_1&B\\\hline
0&0
\end{array}
\right),\quad
\Lparen\theta_s\Rparen=\left( 
\begin{array}{c|c}
0&0\\\hline
B^{\mathrm{tr}}&2E_2
\end{array}
\right).
\end{displaymath}
From Lemma~\ref{lem1701}, we have that
both $BB^{\mathrm{tr}}$ and $B^{\mathrm{tr}}B$ are annihilated by $f_{n}(x)$. 

The main result of the section is the following statement.

\begin{theorem}\label{thm49}
Assume that $W$ is of type $I_2(n)$ with $n=2k>4$ and $n\neq 12,18,30$. 
Then every simple transitive $2$-representation of
$\cS$ is equivalent to either a cell $2$-representation or one of the $2$-representations 
$\mathbf{N}_s^{(n)}$, $\mathbf{N}_t^{(n)}$ constructed in Subsection~\ref{s6.7}. All these
$2$-representations are pairwise non-equivalent.
\end{theorem}

\subsection{Additional simple transitive $2$-representations}\label{s6.7}

In this subsection we just assume that $n=2k>4$.
The left cell $\mathcal{L}_s$ of the element $\theta_s\in\cS$ consists of the elements
$\theta_s$, $\theta_{ts}$, $\theta_{sts}$, $\theta_{tsts}$ and so on, with
$2k-1$ elements in total.
Consider the corresponding cell $2$-representation $\mathbf{C}_{\mathcal{L}_s}$. 
In this subsection we follow closely the approach of \cite[Subsection~5.8]{MaMa} to 
construct, starting from $\mathbf{C}_{\mathcal{L}_s}$, a new simple transitive
$2$-representation of $\cS$, which we later on will denote by $\mathbf{N}_s^{(n)}$.

Set $\mathbf{Q}:=\mathbf{C}_{\mathcal{L}_s}$ and consider $\overline{\mathbf{Q}}$. 
Let $P_w$, $w\in \mathtt{L}:=\{s,ts,sts,tsts,\dots\}$, be representatives of the 
isomorphism classes of the indecomposable projective objects in $\overline{\mathbf{Q}}(\mathtt{i})$ 
and $L_w$, $w\in\mathtt{L}$, be the respective simple tops. Note that $|\mathtt{L}|$ is odd.
With respect to this  choice of a basis, we have
\begin{displaymath}
\Lparen\theta_s\Rparen=
\left(
\begin{array}{cccccc}
2&1&0&0&0&\dots\\ 
0&0&0&0&0&\dots\\ 
0&1&2&1&0&\dots\\ 
0&0&0&0&0&\dots\\ 
0&0&0&1&2&\dots\\ 
\vdots&\vdots&\vdots&\vdots&\vdots&\ddots
\end{array}
\right)\qquad
\Lparen\theta_t\Rparen=
\left(
\begin{array}{cccccc}
0&0&0&0&0&\dots\\ 
1&2&1&0&0&\dots\\ 
0&0&0&0&0&\dots\\ 
0&0&1&2&1&\dots\\ 
0&0&0&0&0&\dots\\
\vdots&\vdots&\vdots&\vdots&\vdots&\ddots
\end{array}
\right).
\end{displaymath}
In particular, in this case the matrix $B$ is a $k\times (k-1)$ staircase matrix.
Consequently, we have the following Loewy filtrations of indecomposable projective modules:
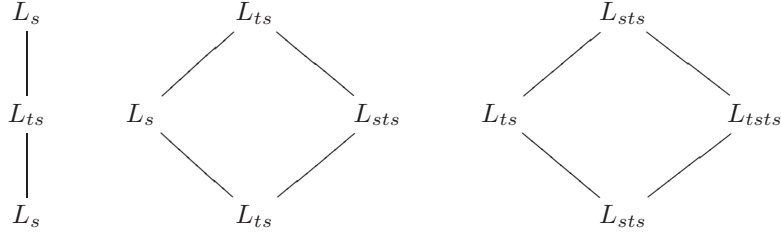
\begin{displaymath}
\xymatrix{
L_s\ar@{-}[d]&&L_{ts}\ar@{-}[dl]\ar@{-}[dr]&&&L_{sts}\ar@{-}[dl]\ar@{-}[dr]&\\
L_{ts}\ar@{-}[d]&L_s\ar@{-}[dr]&&L_{sts}\ar@{-}[dl]&L_{ts}\ar@{-}[dr]&&L_{tsts}\ar@{-}[dl]\\
L_s&&L_{ts}&&&L_{sts}&\\
}
\end{displaymath}
and so on (the last module in the series is uniserial of length three, just like the first one).
Let $H$ denote the 
basic underlying algebra of $\overline{\mathbf{Q}}(\mathtt{i})$. Then the above implies that
$H$ is the quotient of the path algebra of 
\begin{displaymath}
\xymatrix{
\mathtt{s}\ar@/^/[r]&\mathtt{ts}\ar@/^/[r]\ar@/^/[l]&\mathtt{sts}\ar@/^/[r]\ar@/^/[l]&
\mathtt{tsts}\ar@/^/[r]\ar@/^/[l]&\dots\ar@/^/[l]
}
\end{displaymath}
modulo the relations that any path of the form $\mathtt{v}\to\mathtt{u}\to \mathtt{w}$ is zero
if $v\neq w$ and all paths of the form $\mathtt{v}\to\mathtt{u}\to \mathtt{v}$ coincide.
Let $f_w$, $w\in\mathtt{L}$, denote pairwise orthogonal primitive idempotents of $H$ corresponding 
to $P_w$.

As $\mathbf{Q}$ is simple transitive, all $\theta_w$, $w\in\mathcal{J}$, send 
simple objects in $\overline{\mathbf{Q}}(\mathtt{i})$ to
projective objects in $\overline{\mathbf{Q}}(\mathtt{i})$
and act as projective endofunctors of $\overline{\mathbf{Q}}(\mathtt{i})$, by Theorem~\ref{thmminimal}.
In terms of $H$, the action of $\theta_s$ is given by tensoring with the $H$-$H$--bimodule
\begin{displaymath}
(Hf_{s}\otimes f_{s}H)\oplus(Hf_{sts}\otimes f_{sts}H)\oplus\dots  
\end{displaymath}
while the action of $\theta_{t}$ is, similarly,  given by tensoring with the $H$-$H$--bimodule
\begin{displaymath}
(Hf_{ts}\otimes f_{ts}H)\oplus(Hf_{tsts}\otimes f_{tsts}H)\oplus\dots. 
\end{displaymath}

Set $\mathbf{Q}^{(0)}:=\mathbf{Q}$ and let $\mathbf{Q}^{(1)}$ denote the $2$-representation of $\cS$
given by the action of $\cS$ on the category of projective objects in 
$\overline{\mathbf{Q}}^{(0)}(\mathtt{i})$. Recursively, for $k\geq 1$, define 
$\mathbf{Q}^{(k)}$ as the $2$-representation of $\cS$ given by the action of $\cS$ 
on the category of projective objects in $\overline{\mathbf{Q}}^{(k-1)}(\mathtt{i})$.
For every $k\geq 0$, we have a strict $2$-natural transformation 
$\Lambda_k:{\mathbf{Q}}^{(k-1)}\to \mathbf{Q}^{(k)}$ which sends an object $X$ to the diagram
$0\to X$ and a morphism $\alpha:X\to X'$ to the diagram 
\begin{displaymath}
\xymatrix{
0\ar[rr]\ar[d] && X\ar[d]^{\alpha}\\
0\ar[rr] && X'.
} 
\end{displaymath}
Clearly, each such $\Lambda_k$ is an equivalence. 
Denote by $\mathbf{K}$ the inductive limit of the directed system
\begin{equation}\label{eq5z}
\mathbf{Q}^{(0)}\overset{\Lambda_0}{\longrightarrow}
\mathbf{Q}^{(1)}\overset{\Lambda_1}{\longrightarrow}
\mathbf{Q}^{(2)}\overset{\Lambda_2}{\longrightarrow}\dots. 
\end{equation}
Then $\mathbf{K}$ is a $2$-representation of $\cS$ which is equivalent to $\mathbf{Q}$.

Let $x'$ be the element which you get by swapping $s$ and $t$ in the reduced expression for $x$.

\begin{lemma}\label{lem43}
There is a strict $2$-natural transformation $\Psi:\mathbf{K}\to \mathbf{K}$ which is an equivalence 
and which acts on the isomorphism classes of the indecomposable projective objects by swapping
each $P_x$ with $P_{x'w_0}$, for $x\in\mathtt{L}$ such that $x'w_0\neq x$, and fixing 
the unique $P_{x}$ for which $x'w_0=x$. 
\end{lemma}

\begin{proof}
Consider the unique strict $2$-natural transformation $\Phi:\mathbf{P}_{\mathtt{i}}\to \overline{\mathbf{Q}}$
which sends $\mathbb{1}_{\mathtt{i}}$ to $L_{tw_0}$. Using Proposition~\ref{propn31} and the explicit 
formulae for the matrices representing the action of all $\theta_w$, we have
\begin{displaymath}
\theta_s\, L_{tw_0}\cong P_{tw_0},\quad
\theta_{ts}\, L_{tw_0}\cong P_{stw_0},\quad
\text{ and so on}.
\end{displaymath}
As the Cartan matrix of $H$ is invariant under swapping $P_x$ with $P_{x'w_0}$, 
for $x\in\mathtt{L}$ such that $x'w_0\neq x$, and fixing 
the unique $P_{x}$ for which $x'w_0=x$, it follows by the usual arguments, see for example \cite[Subsection~4.9]{MaMa}, 
that the $2$-natural transformation $\Phi$ factors through $\mathbf{C}_{\mathcal{L}_s}$ and, therefore,  
gives a strict equivalence $\Phi^{(0)}:\mathbf{Q}^{(0)}\to \mathbf{Q}^{(1)}$. 
For $k\geq 0$, via abelianization, we get a strict equivalence 
$\Phi^{(k)}:\mathbf{Q}^{(k)}\to \mathbf{Q}^{(k+1)}$, which is
compatible with \eqref{eq5z}. Now we can take $\Psi$ as the inductive limit of $\Phi^{(k)}$.
\end{proof}

Consider a new finitary $2$-representation $\mathbf{K}'$ of $\cS$ defined as follows:
\begin{itemize}
\item Objects of $\mathbf{K}'(\mathtt{i})$ are sequences 
$(X_n,\alpha_n)_{n\in\mathbb{Z}}$, where $X_n$ is an object in 
$\mathbf{K}(\mathtt{i})$ and $\alpha_n\colon \Psi(X_n)\to X_{n+1}$ an isomorphism 
in $\mathbf{K}(\mathtt{i})$, for all $n\in\mathbb{Z}$. 
\item Morphisms in $\mathbf{K}'(\mathtt{i})$ from $(X_n,\alpha_n)_{n\in\mathbb{Z}}$
to $(Y_n,\beta_n)_{n\in\mathbb{Z}}$ are sequences of morphisms $f_n\colon X_n\to Y_n$ 
in $\mathbf{K}(\mathtt{i})$ such that   
\begin{displaymath}
\xymatrix{
\Psi(X_n)\ar[rr]^{\alpha_n}\ar[d]_{\Psi(f_n)}&&X_{n+1}\ar[d]^{f_{n+1}}\\
\Psi(Y_n)\ar[rr]^{\beta_n}&&Y_{n+1}\\
}
\end{displaymath}
commutes for all $n\in\mathbb{Z}$. 
\item The action of $\cS$ on $\mathbf{K}'(\mathtt{i})$ is inherited from the
action of $\cS$ on $\mathbf{K}(\mathtt{i})$ component-wise.
\end{itemize}
The construction of $\mathbf{K}'(\mathtt{i})$ from $\mathbf{K}(\mathtt{i})$ is the
standard construction which turns a category with an autoequivalence (in our case
$\Psi$) into an equivalent category with an automorphism, cf. \cite{Ke,BL}. 

We have the strict $2$-natural transformation $\Pi:\mathbf{K}'\to \mathbf{K}$ given by
projection onto the zero component of a sequence. This $\Pi$ is an equivalence, by construction.
We also have a strict $2$-natural transformation $\Psi':\mathbf{K}'\to \mathbf{K}'$
given by shifting the entries of the sequences by one, that is, sending 
$(X_n,\alpha_n)_{n\in\mathbb{Z}}$ to $(X_{n+1},\alpha_{n+1})_{n\in\mathbb{Z}}$,
with the similar obvious action on morphisms. Note that 
$\Psi':\mathbf{K}'(\mathtt{i})\to \mathbf{K}'(\mathtt{i})$ is an 
automorphism. The functor $\Psi'$ acts on the isomorphism classes of indecomposable 
objects in $\mathbf{K}'(\mathtt{i})$ in the same way as $\Psi$ does.
As $\Psi^2$ is isomorphic to the identity functor on 
$\mathbf{K}(\mathtt{i})$, it follows, by construction, that $(\Psi')^2$ is isomorphic to 
the identity functor on $\mathbf{K}'(\mathtt{i})$. 

Let $\mathbf{K}''$ be the $2$-representation of $\cS$ given by the action of $\cS$ on the
category of projective objects in $\overline{\mathbf{K}'}$. We denote by $\Psi'':\mathbf{K}''\to \mathbf{K}''$
the diagrammatic extension of $\Psi'$ to $\mathbf{K}''$. Again, 
$\Psi'':\mathbf{K}''(\mathtt{i})\to \mathbf{K}''(\mathtt{i})$ is an automorphism and
$(\Psi'')^2$ is isomorphic to  the identity functor on $\mathbf{K}''(\mathtt{i})$. 
We need the following stronger statement.

\begin{lemma}\label{lem44}
Let $\mathrm{Id}:\mathbf{K}''\to \mathbf{K}''$ denote the identity 
$2$-natural transformation.
\begin{enumerate}[$($i$)$]
\item\label{lem44.1} 
There is an invertible modification $\eta:\mathrm{Id}\to (\Psi'')^2$.
\item\label{lem44.2}
For any $\eta$ as in \eqref{lem44.1}, we have
$\mathrm{id}_{(\Psi'')^2}\circ_0 \eta=\eta\circ_0\mathrm{id}_{(\Psi'')^2}$.
\item\label{lem44.3}
For any $\eta'\in\mathrm{Hom}_{\ccS\text{-}\mathrm{mod}}(\mathrm{Id},(\Psi'')^2)$, we have
$\mathrm{id}_{(\Psi'')^2}\circ_0 \eta'=\eta'\circ_0\mathrm{id}_{(\Psi'')^2}$.
\end{enumerate}
\end{lemma}

\begin{proof}
Let $L$ be a simple object in $\overline{\mathbf{K}'}$ corresponding to $P_1$
(thus $L$ is a simple corresponding to the Duflo involution in $\mathcal{L}_s$). Fix an  isomorphism
$\alpha:L\to (\Psi')^2(L)$. For $w\in \mathtt{L}$, set $\eta_{\theta_w\,L}:=\mathrm{id}_{\theta_w}\circ_0\alpha$.
As $\{\theta_w\,L\,:\,w\in\mathtt{L}\}$ is a complete list of pairwise non-isomorphic indecomposable
objects in $\mathbf{K}''(\mathtt{i})$, this uniquely defines a natural transformation 
$\eta:\mathrm{Id}\to (\Psi'')^2$.

If $\mathrm{F}$ is a $1$-morphism in $\cS$, then, for any $w\in \mathtt{L}$, we have 
$\mathrm{F}\circ \theta_w\cong \mathrm{F}_1\oplus \mathrm{F}_2$, where $\mathrm{F}_1\, L\in
\mathrm{add}(\{\theta_w\,L\,:\,w\in\mathtt{L}\})$ and $\mathrm{F}_2\, L=0$. From the definition in the
previous paragraph we thus get $\mathrm{F}(\eta_{\theta_w\,L})=\eta_{\mathrm{F}_1\,L}$ which implies
that $\eta$ is, in fact, a modification. Claim~\eqref{lem44.1} follows. Claims~\eqref{lem44.2}
and \eqref{lem44.3} are now proved similarly to \cite[Lemma~17(2) and (3)]{MaMa}.
\end{proof}
 
\begin{proposition}\label{prop45}
There is an invertible modification $\eta:\mathrm{Id}\to (\Psi'')^2$ 
for which we have the equality $\mathrm{id}_{\Psi''}\circ_0 \eta=\eta\circ_0\mathrm{id}_{\Psi''}$.
\end{proposition}

\begin{proof}
Mutatis mutandis the proof of \cite[Proposition~18]{MaMa}.
\end{proof}

From now on we fix some invertible modification $\eta:\mathrm{Id}\to (\Psi'')^2$
as given by Proposition~\ref{prop45}. Define a small category $\mathbf{L}(\mathtt{i})$ as follows:
\begin{itemize}
\item objects  in $\mathbf{L}(\mathtt{i})$  are all $4$-tuples $(X,Y,\alpha,\beta)$, where we have
$X,Y\in \mathbf{K}''(\mathtt{i})$, while $\alpha:X\to \Psi''(Y)$ and $\beta:Y\to \Psi''(X)$ are 
isomorphism such that the following conditions are satisfied
\begin{equation}\label{eqnn9}
\left\{ 
\begin{array}{rcl} 
\eta_{Y}^{-1}\circ_1 {\Psi''}(\alpha)\circ_1\beta&=&\mathrm{id}_Y,\\
\beta\circ_1\eta_{Y}^{-1}\circ_1 {\Psi''}(\alpha)&=&\mathrm{id}_{\Psi''(X)},\\ 
\eta_{X}^{-1}\circ_1{\Psi''}(\beta)\circ_1\alpha&=&\mathrm{id}_X,\\
\alpha\circ_1\eta_{X}^{-1}\circ_1{\Psi''}(\beta)&=&\mathrm{id}_{\Psi''(Y)}. 
\end{array}
\right.
\end{equation}
\item morphisms in $\mathbf{L}(\mathtt{i})$ from $(X,Y,\alpha,\beta)$ to $(X',Y',\alpha',\beta')$
are pairs $(\zeta,\xi)$, where $\zeta:X\to X'$ and $\xi:Y\to Y'$ are morphisms in 
$\mathbf{K}''(\mathtt{i})$ such that the diagrams
\begin{displaymath} 
\xymatrix{
X\ar[rr]^{\alpha}\ar[d]_{\zeta}&&\Psi''(Y)\ar[d]^{\Psi''(\xi)}\\
X'\ar[rr]^{\alpha'}&&\Psi''(Y')\\
} \quad\text{ and }\quad
\xymatrix{
Y\ar[rr]^{\beta}\ar[d]_{\xi}&&\Psi''(X)\ar[d]^{\Psi''(\zeta)}\\
Y'\ar[rr]^{\beta'}&&\Psi''(X')\\
} 
\end{displaymath}
commute;
\item the composition and identity morphisms are the obvious ones.
\end{itemize}
The category $\mathbf{L}(\mathtt{i})$ comes equipped with an action of $\cS$, defined component-wise,
using the action of $\cS$ on $\mathbf{K}(\mathtt{i})$. This is well-defined as the
$2$-natural transformation $\Psi$ is strict by Lemma~\ref{lem43} and, moreover, $\eta$ is a modification. 
We denote the corresponding $2$-representation of $\cS$ by $\mathbf{L}$. 

\begin{lemma}\label{lem46}
Restriction to the first component of a quadruple defines a strict $2$-natural 
transformation  $\Upsilon:\mathbf{L}\to \mathbf{K}''$. This $\Upsilon$ is an equivalence, 
\end{lemma}

\begin{proof}
Mutatis mutandis the proof of \cite[Lemma~19]{MaMa}.
%
\end{proof}

Define an endofunctor $\Theta$ on $\mathbf{L}(\mathtt{i})$ by sending $(X,Y,\alpha,\beta)$ to 
$(Y,X,\beta,\alpha)$ with the obvious action on morphisms. From all symmetries in  the 
definition of $\mathbf{L}(\mathtt{i})$, it follows that $\Theta$ is a strict involution 
and it also strictly commutes with the action of $\cS$.

Next, consider the category $\mathbf{L}'(\mathtt{i})$ defined as follows:
\begin{itemize}
\item $\mathbf{L}'(\mathtt{i})$ has the same objects as $\mathbf{L}(\mathtt{i})$,
\item morphisms in  $\mathbf{L}'(\mathtt{i})$ are defined, for objects $X,Y\in\mathbf{L}(\mathtt{i})$, via
\begin{displaymath}
\mathrm{Hom}_{\mathbf{L}'(\mathtt{i})}(X,Y):=
\mathrm{Hom}_{\mathbf{L}(\mathtt{i})}(X,Y)\oplus  \mathrm{Hom}_{\mathbf{L}(\mathtt{i})}(X,\Theta(Y)),
\end{displaymath}
\item composition and identity morphisms in $\mathbf{L}'(\mathtt{i})$ are induced from those
in $\mathbf{L}(\mathtt{i})$ in the obvious way, see \cite[Definition~2.3]{CiMa} for details.
\end{itemize}
The fact that $\Psi$ preserves the isomorphism class of the indecomposable object $P_{sts}$ 
implies that the endomorphism algebra of the corresponding object in $\mathbf{L}'(\mathtt{i})$
contains a copy of the group algebra of the group $\{\mathrm{Id},\Theta\}$ and hence is not local.
This means that $\mathbf{L}'(\mathtt{i})$ is not idempotent split.
Denote by $\mathbf{N}_s^{(n)}(\mathtt{i})$ the idempotent completion of $\mathbf{L}'(\mathtt{i})$.

\begin{proposition}\label{prop47}
{\hspace{2mm}}

\begin{enumerate}[$($i$)$]
\item\label{prop27.1} The category  $\mathbf{N}_s^{(n)}(\mathtt{i})$ is a finitary $\mathbb{C}$-linear category.
\item\label{prop27.2} The category  $\mathbf{N}_s^{(n)}(\mathtt{i})$ is equipped with an 
action of $\cS$ induced from that on $\mathbf{L}'(\mathtt{i})$.
\item\label{prop27.3} The  obvious functor
$\Xi:\mathbf{L}(\mathtt{i})\to \mathbf{N}_s^{(n)}(\mathtt{i})$ is a strict $2$-natural transformation.
\end{enumerate}
\end{proposition}

\begin{proof}
The only thing which is different from  \cite[Proposition~20]{MaMa} is the fact that the
category $\mathbf{L}'(\mathtt{i})$ is not idempotent split. However, since we define 
$\mathbf{N}_s^{(n)}(\mathtt{i})$ as the idempotent completion of $\mathbf{L}'(\mathtt{i})$,
it follows that $\mathbf{N}_s^{(n)}(\mathtt{i})$ is idempotent split and hence 
$\mathbf{N}_s^{(n)}(\mathtt{i})$ is a finitary $\mathbb{C}$-linear category.
The rest is similar to \cite[Proposition~20]{MaMa}.
%
\end{proof}

The above gives us the $2$-representation $\mathbf{N}_s^{(n)}$ of $\cS$. 
From the fact that $\mathbf{C}_{\mathcal{L}_s}$ is simple transitive, it follows
that $\mathbf{N}_s^{(n)}$ is simple transitive.
The underlying algebra $\mathbf{N}_s^{(n)}$ is the quotient of the path algebra 
of the following  quiver:
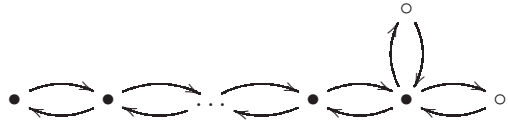
\begin{displaymath}
\xymatrix{
&&&&\circ\ar@/^/[d]&\\
\bullet\ar@/^/[r]&\bullet\ar@/^/[r]\ar@/^/[l]&\dots\ar@/^/[r]\ar@/^/[l]&\bullet\ar@/^/[r]\ar@/^/[l]
&\bullet\ar@/^/[r]\ar@/^/[l]\ar@/^/[u]&\circ\ar@/^/[l]
}
\end{displaymath}
where we mod out by the relations that any path of the form $\mathtt{i}\to\mathtt{j}\to \mathtt{k}$ is zero
if $i\neq k$ and all paths of the form $\mathtt{i}\to\mathtt{j}\to \mathtt{i}$ coincide.
Here the vertices $\bullet$ correspond to pairs $\{P_x,P_{x'w_0}\}$, where $x\in\mathtt{L}$
is such that $x\neq x'w_0$, while the two vertices $\circ$ correspond to the unique 
$x\in\mathtt{L}$ for which $x'w_0=x$. 

If $k$ is odd, then the unique $x\in\mathtt{L}$ for which $x'w_0=x$ has the form $stst\dots s$.
This implies that the matrix $B$ is, up to permutation of rows and columns, 
an extended staircase matrix of size $\frac{k+3}{2}\times \frac{k-1}{2}$.
If $k$ is even, then the unique $x\in\mathtt{L}$ for which $x'w_0=x$ has the form $tst\dots s$.
This implies that the matrix $B$ is, up to permutation of rows and columns, 
an extended staircase matrix of size $\frac{k}{2}\times \frac{k+2}{2}$. The explanation why
$B$ has only one row or column with three non-zero elements is the fact that the quiver above
has only one vertex which is connected to three other vertices.

We denote by $\mathbf{N}_t^{(n)}$ the $2$-representation of $\cS$ constructed similarly 
starting from $\mathcal{L}_t$. Note that $\mathbf{C}_{\mathcal{L}_s}$ and
$\mathbf{C}_{\mathcal{L}_t}$ are not equivalent as their decategorifications contain
different one-dimensional simple $W$-modules and hence are not isomorphic. 
It is easy to check that the decategorifications of $\mathbf{N}_s^{(n)}$ and
$\mathbf{N}_t^{(n)}$ also contain different one-dimensional simple $W$-modules. 
This implies that $\mathbf{N}_s^{(n)}$ and
$\mathbf{N}_t^{(n)}$ are not equivalent. Comparing decategorifications, we, in fact, see
that the $2$-representation $\mathbf{C}_{\mathcal{L}_s}$, $\mathbf{C}_{\mathcal{L}_t}$,
$\mathbf{N}_s^{(n)}$ and $\mathbf{N}_t^{(n)}$ are pairwise not equivalent.

\subsection{The matrix $B$}\label{s6.2}

Our aim in this section is to describe the matrix $B$ from Subsection~\ref{s6.1}.

\begin{proposition}\label{prop42}
Let $\mathbf{M}$ be  a simple transitive $2$-representation of $\underline{\cS}$ with apex $\mathcal{J}$.
Assume that $n\neq 12,18,30$.
Then, up to permutation of rows and columns,  the matrix $B$ is a staircase matrix
of size $k\times (k-1)$ or an extended staircase matrix 
of size $\frac{k}{2}\times \frac{k+2}{2}$ (if $k$ is even) 
or of size $\frac{k+3}{2}\times \frac{k-1}{2}$ (if $k$ is odd).
\end{proposition}

For example, if $n=6$,  then $B$ is one of the following matrices:
\begin{displaymath}
\left(\begin{array}{ccc}1&1&0\\0&1&1\end{array}\right),\quad 
\left(\begin{array}{cc}1&0\\1&1\\0&1\end{array}\right),\quad
\left(\begin{array}{ccc}1&1&1\end{array}\right),\quad 
\left(\begin{array}{c}1\\1\\1\end{array}\right).
\end{displaymath}
If $n=8$,  then $B$ is one of the following matrices:
\begin{displaymath}
\left(\begin{array}{cccc}1&1&0&0\\0&1&1&0\\0&0&1&1\end{array}\right),\quad 
\left(\begin{array}{ccc}1&0&0\\1&1&0\\0&1&1\\0&0&1\end{array}\right),\quad
\left(\begin{array}{ccc}1&1&1\\0&0&1\end{array}\right),\quad 
\left(\begin{array}{cc}1&0\\1&0\\1&1\end{array}\right). 
\end{displaymath}

\begin{proof}
As $\mathbf{M}$ is simple transitive, the non-negative matrix $\mathtt{M}$ is primitive.
From the computation in the proof of Lemma~\ref{lem1701} we thus get that both 
$BB^{\mathrm{tr}}$ and $B^{\mathrm{tr}}B$ must be irreducible non-negative matrices.
The combination of  Lemma~\ref{lem1701} with Subsection~\ref{s17.2} implies that both
$BB^{\mathrm{tr}}$ and $B^{\mathrm{tr}}B$ are diagonalizable with real eigenvalues, moreover,
all eigenvalues are contained in the half-open interval $[0,4)$.

From Proposition~\ref{propsim}, we thus obtain that $B$ is either a staircase matrix
or an extended staircase matrix or coincides with $X_1$ or $X_2$ or $X_3$.
Note that all staircase and extended staircase matrices appear as $B$ for some type
$I_2(l)$, where $l$ can be arbitrary. Therefore, thanks to Corollary~\ref{cor1705}, to 
complete the proof of our proposition, it is enough to show that $B$ cannot coincide
with any of $X_1,X_2$ or $X_3$.

The minimal polynomial of the matrix $X_1X_1^{\mathrm{tr}}$ is $(x-1)(x^2-4x+1)=
\underline{f}_{3}(x)\underline{f}_{12}(x)$.
The arguments in the proof of Corollary~\ref{cor1705} imply that $B$ can be equal to 
$X_1$ only in the case $n=12$.

The minimal polynomial of the matrix $X_2X_2^{\mathrm{tr}}$ is $x^3-6x^2+9x-3=\underline{f}_{18}(x)$.
The arguments in the proof of Corollary~\ref{cor1705} imply  that $B$ can be equal to 
$X_1$ only in the case $n=18$.

The minimal polynomial of the matrix $X_3X_3^{\mathrm{tr}}$ is $x^4-7x^3+14x^2-8x+1=\underline{f}_{30}(x)$.
The arguments in the proof of Corollary~\ref{cor1705} imply  that $B$ can be equal to 
$X_1$ only in the case $n=30$.

As the cases $n=12,18,30$ are excluded, the claim of the proposition follows.
\end{proof}

\subsection{Proof of Theorem~\ref{thm49}}\label{s6.8}

Let $\mathbf{M}$ be a simple transitive $2$-representation of $\cS$. 
Since $n\neq 12,18,30$, we can apply Proposition~\ref{prop42} to get 
four possibilities for $B$ which are in a natural bijection with 
the $2$-representations $\mathbf{C}_{\mathcal{L}_s}$,  $\mathbf{C}_{\mathcal{L}_t}$,
$\mathbf{N}_s^{(n)}$ and $\mathbf{N}_t^{(n)}$. That in the first two cases we have that $\mathbf{M}$
is equivalent to $\mathbf{C}_{\mathcal{L}_s}$ or, respectively,  $\mathbf{C}_{\mathcal{L}_t}$,
is proved similarly to Subsection~\ref{s17.6}. That in the last two cases we have that $\mathbf{M}$
is equivalent to $\mathbf{N}_s^{(n)}$ or, respectively,  $\mathbf{N}_t^{(n)}$,
is proved similarly to \cite[Subsection~5.10]{MaMa}. 

\subsection{Exceptional types $I_2(12)$, $I_2(18)$ and $I_2(30)$}\label{s6.9}

In Coxeter type $I_2(12)$, the proof of Proposition~\ref{prop42} leaves the possibility of 
$B=X_1$. In this case 
\begin{displaymath}
\mathtt{M}=\left(\begin{array}{cccccc}2&0&0&1&0&0\\0&2&0&1&1&1\\0&0&2&0&0&1
\\1&1&0&2&0&0\\0&1&0&0&2&0\\0&1&1&0&0&2\end{array}\right). 
\end{displaymath}
If we assume that a simple transitive $2$-representation with such $\mathtt{M}$ exists, then the
underlying algebra of this $2$-representation must be the quotient of the path algebra of the following 
quiver (here the number of each vertex corresponds to the numbering of columns in $\mathtt{M}$):
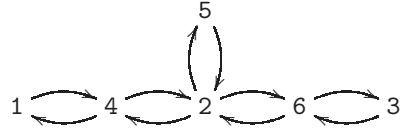
\begin{displaymath}
\xymatrix{
&&\mathtt{5}\ar@/^/[d]&&&&\\
\mathtt{1}\ar@/^/[r]&\mathtt{4}\ar@/^/[r]\ar@/^/[l]&\mathtt{2}\ar@/^/[r]\ar@/^/[l]\ar@/^/[u]&
\mathtt{6}\ar@/^/[r]\ar@/^/[l]&\mathtt{3}\ar@/^/[l]
}
\end{displaymath}
modulo the relations that any path of the form $\mathtt{i}\to\mathtt{j}\to \mathtt{k}$ is zero
if $i\neq k$ and all paths of the form $\mathtt{i}\to\mathtt{j}\to \mathtt{i}$ coincide.
This algebra is the quadratic dual of the preprojective algebra of the underlying Dynkin quiver
of type $E_6$. This suggests a relation between type $I_2(12)$ and type $E_6$ which we do not
understand. We do not know whether this hypothetical $2$-representation exists and we see no reasons
why it should not exist. On the decategorified level, the corresponding representation of the
group algebra, on which the elements of the Kazhdan-Lusztig basis act via the corresponding 
non-negative matrices certainly does exist. Because of Subsection~\ref{s7.3}, 
there is also a possibility of connection between type $I_2(12)$ and type $F_4$.

Similarly, Coxeter types $I_2(18)$  and $I_2(30)$ are connected to Dynkin types $E_7$ and $E_8$, respectively.
Because of Subsection~\ref{s7.2}, Coxeter type $I_2(30)$ could also be connected to Coxeter type $H_4$.

\section{Simple transitive $2$-representations of $\underline{\cS}$ in other Coxeter types
of rank higher than two}\label{s7}

\subsection{Coxeter type $H_3$}\label{s7.1}

In this subsection we assume that $W$ is of Coxeter type $H_3$ and $S=\{r,s,t\}$ with the Coxeter diagram
\begin{displaymath}
\xymatrix{r\ar@{-}[r]^5&s\ar@{-}[r]&t.} 
\end{displaymath}

\begin{proposition}\label{prop81}
Assume that $W$ is of Coxeter type $H_3$. Let $\mathbf{M}$ be a simple transitive $2$-representation of
$\underline{\cS}$ with apex $\mathcal{J}$. Then the isomorphism classes of indecomposable objects in
$\mathbf{M}(\mathtt{i})$ can be ordered such that, with respect to the ordering $r,s,t$, we have
\begin{displaymath}
\mathtt{M}=\left(
\begin{array}{cccccc}
2&0&1&0&0&0\\
0&2&1&1&0&0\\
1&1&2&0&1&0\\
0&1&0&2&0&1\\
0&0&1&0&2&0\\
0&0&0&1&0&2
\end{array}
\right).
\end{displaymath}
\end{proposition}

\begin{proof}
To prove this statement we use reduction to rank two Coxeter subgroups. 
Let $\cC_1$ denote the $2$-full $2$-subcategory of $\underline{\cS}$ which 
is monoidally generated by $\theta_r$ and $\theta_s$. By construction,
the $2$-category $\cC_1$ has a quotient which is biequivalent to 
the  $2$-category $\underline{\cS}$ in Coxeter type $I_2(5)$.
Let $\cC_2$ denote the $2$-full $2$-subcategory of $\underline{\cS}$ which 
is monoidally generated by $\theta_s$ and $\theta_t$. By construction,
the $2$-category  $\cC_2$ has a quotient which is biequivalent to 
the  $2$-category $\underline{\cS}$ in  Coxeter type $A_2$.

We can restrict the action of $\cC_1$ to the additive closure $\mathcal{X}$, in $\mathbf{M}(\mathtt{i})$,
of all indecomposable objects which are not annihilated by either $\theta_r$ or $\theta_s$.
From Theorem~\ref{thm1700}, $\cC_1$ has a unique (up to equivalence) simple transitive 
$2$-representation which does not annihilate any non-zero $1$-morphisms. In this 
$2$-representation, the actions of $\theta_r$ and $\theta_s$ are given by the following
matrices:
\begin{equation}\label{eq81-1}
\left(
\begin{array}{cccc}
2&0&1&0\\
0&2&1&1\\
0&0&0&0\\
0&0&0&0
\end{array}
\right),\qquad 
\left(
\begin{array}{cccc}
0&0&0&0\\
0&0&0&0\\
1&1&2&0\\
0&1&0&2
\end{array}
\right).
\end{equation}
By the combination of \cite[Corollary~20]{KM2} and \cite[Theorem~25]{CM}, the actions of
$\theta_r$ and $\theta_s$ on $\mathcal{X}$ are then given by a direct sum of blocks of the
form \eqref{eq81-1}.

We can restrict the action of $\cC_2$ to the additive closure $\mathcal{Y}$, in $\mathbf{M}(\mathtt{i})$,
of all indecomposable objects which are not annihilated by either $\theta_s$ or $\theta_t$.
From \cite[Theorem~18]{MM5}, $\cC_2$ has a unique (up to equivalence) simple transitive 
$2$-representation which does not annihilate any non-zero $1$-morphisms. In this 
$2$-representation, the actions of $\theta_s$ and $\theta_t$ are given by the following
matrices:
\begin{equation}\label{eq81-2}
\left(
\begin{array}{cc}
2&1\\
0&0
\end{array}
\right)\qquad 
\left(
\begin{array}{cc}
0&0\\
1&2
\end{array}
\right).
\end{equation}
By the combination of \cite[Corollary~20]{KM2} and \cite[Theorem~25]{CM}, the actions of
$\theta_s$ and $\theta_t$ on $\mathcal{Y}$ are then given by a direct sum of blocks of the
form \eqref{eq81-2}.

We need to combine blocks of the form \eqref{eq81-1} with blocks of the form \eqref{eq81-2} 
to get an irreducible non-negative matrix for the principal element $\mathbf{s}$, 
of size $3m$, for some $m$. The latter restriction is due to the fact that 
the decategorification of $\mathbf{M}$ is a sum of $3$-dimensional simple $W$-modules. 
Note that \eqref{eq81-2} allows us to ``connect'' (in the sense of having a non-zero element 
on the intersection of the corresponding row and column) one indecomposable which is not 
annihilated by $\theta_s$ with one indecomposable which is not annihilated by $\theta_t$. 
Therefore the only way is to connect each of the two indecomposables in some block of the form 
\eqref{eq81-1} which are not annihilated by $\theta_s$ with an external pair of two 
indecomposables which are not annihilated by $\theta_t$. This gives exactly  the matrices 
given in the formulation of the proposition (up to permutation of basis elements).
The claim of the proposition follows.
\end{proof}

\begin{theorem}\label{thm82}
Assume that $W$ is of Coxeter type $H_3$. Then every simple transitive $2$-representation of 
$\underline{\cS}$ is equivalent to a cell $2$-representation.
\end{theorem}

\begin{proof}
Let $\mathbf{M}$ be a simple transitive $2$-representation of $\underline{\cS}$ with apex $\mathcal{J}$.
We assume that the equivalence classes of indecomposable objects in $\mathbf{M}(\mathtt{i})$ are ordered
such that the decategorification matrices of $\theta_r$, $\theta_s$ and $\theta_t$ are gives as
in Proposition~\ref{prop81}. Consider $\overline{\mathbf{M}}$ and the $2$-representation $\mathbf{N}$
of $\underline{\cS}$ given by restriction of the action to the category of projective objects in 
$\overline{\mathbf{M}}(\mathtt{i})$. Then $\mathbf{M}$ and $\mathbf{N}$ are equivalent by 
\cite[Theorem~11]{MM2}. We call indecomposable projectives in $\overline{\mathbf{M}}(\mathtt{i})$,
in order, $P_1,P_2,\dots,P_6$ and their corresponding simple tops $L_1,L_2,\dots,L_6$. Using the 
explicit matrices given by Proposition~\ref{prop81} and arguments similar to the one used in
Subsection~\ref{s17.6}, one shows that 
\begin{gather*}
\theta_r\,L_1\cong P_1,\quad 
\theta_{sr}\,L_1\cong P_3,\quad
\theta_{rsr}\,L_1\cong P_2,\quad
\theta_{tsr}\,L_1\cong P_5,\\
\theta_{srsr}\,L_1\cong P_4,\quad
\theta_{tsrsr}\,L_1\cong P_6.
\end{gather*}
Now the proof is completed as in Subsection~\ref{s17.6}.
There is a unique strong $2$-natural transformation from $\mathbf{P}_{\mathtt{i}}$ to 
$\overline{\mathbf{M}}$ sending $\mathbbm{1}_{\mathtt{i}}$ to $L_1$. It induces 
a strong $2$-natural transformation from $\mathbf{C}_{\mathcal{L}_s}$ to $\mathbf{N}$.
Comparing the Cartan matrices, we see that the latter $2$-natural transformation
is, in fact, an equivalence. Therefore $\mathbf{C}_{\mathcal{L}_s}$ and $\mathbf{M}$ 
are equivalent.
\end{proof}

\begin{remark}\label{rem82-1}
{\rm
The underlying algebra of a cell $2$-representation of $\underline{\cS}$ with apex $\mathcal{J}$
is the quotient of the path algebra of the following quiver
(here the number $\mathtt{i}$ corresponds to $P_i$):
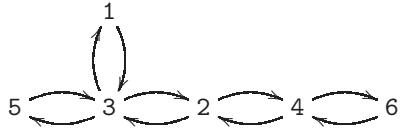
\begin{displaymath}
\xymatrix{
&\mathtt{1}\ar@/^/[d]&&&\\
\mathtt{5}\ar@/^/[r]&\mathtt{3}\ar@/^/[r]\ar@/^/[l]\ar@/^/[u]&
\mathtt{2}\ar@/^/[r]\ar@/^/[l]&\mathtt{4}\ar@/^/[r]\ar@/^/[l]&\mathtt{6}\ar@/^/[l]
}
\end{displaymath}
modulo the relations that any path of the form $\mathtt{i}\to\mathtt{j}\to \mathtt{k}$ is zero
if $i\neq k$ and all paths of the form $\mathtt{i}\to\mathtt{j}\to \mathtt{i}$ coincide.
The Loewy filtrations of the indecomposable projective modules for this algebra are:
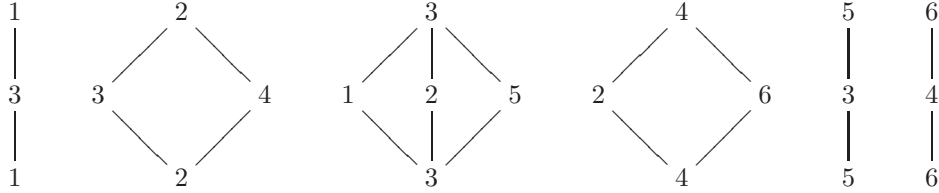
\begin{displaymath}
\xymatrix@!=0.6pc{
1\ar@{-}[d]&&2\ar@{-}[dl]\ar@{-}[dr]&&&3\ar@{-}[d]\ar@{-}[dl]\ar@{-}[dr]
&&&4\ar@{-}[dr]\ar@{-}[dl]&&5\ar@{-}[d]&6\ar@{-}[d]\\
3\ar@{-}[d]&3\ar@{-}[dr]&&4\ar@{-}[dl]&1\ar@{-}[dr]&2\ar@{-}[d]&5\ar@{-}[dl]&2\ar@{-}[dr]&&6\ar@{-}[dl]&3\ar@{-}[d]&4\ar@{-}[d]\\
1&&2&&&3&&&4&&5&6
}
\end{displaymath}
This algebra is the quadratic dual of the preprojective algebra of the underlying Dynkin quiver
of type $D_6$, cf. \cite{Du}. For further connections between these two Coxeter groups,
see \cite{De} and references therein.
} 
\end{remark}

\subsection{Coxeter type $H_4$}\label{s7.2}

In this subsection we assume that $W$ is of Coxeter type $H_4$ and $S=\{r,s,t,u\}$ with the Coxeter  diagram
\begin{displaymath}
\xymatrix{r\ar@{-}[r]^5&s\ar@{-}[r]&t\ar@{-}[r]&u.} 
\end{displaymath}

\begin{proposition}\label{prop83}
Assume that $W$ is of  Coxeter type $H_4$. Let $\mathbf{M}$ be a simple transitive $2$-representation of
$\underline{\cS}$ with apex $\mathcal{J}$. Then the isomorphism classes of indecomposable objects in
$\mathbf{M}(\mathtt{i})$ can be ordered such that, with respect to the ordering
$r,s,t,u$, we have
\begin{displaymath}
\mathtt{M}=\left(
\begin{array}{cccccccc}
2&0&1&0&0&0&0&0\\
0&2&1&1&0&0&0&0\\
1&1&2&0&1&0&0&0\\
0&1&0&2&0&1&0&0\\
0&0&1&0&2&0&1&0\\
0&0&0&1&0&2&0&1\\
0&0&0&0&1&0&2&0\\
0&0&0&0&0&1&0&2
\end{array}
\right).
\end{displaymath}
\end{proposition}

\begin{proof}
Mutatis mutandis the proof of Proposition~\ref{prop81}.
\end{proof}

\begin{theorem}\label{thm84}
Assume that $W$ is of Coxeter  type $H_4$. Then every simple transitive $2$-representation of 
$\underline{\cS}$ is equivalent to a cell $2$-representation.
\end{theorem}

\begin{proof}
Mutatis mutandis the proof of Theorem~\ref{thm82}.
\end{proof}

\begin{remark}\label{rem84-1}
{\rm
The underlying algebra of a cell $2$-representation of $\underline{\cS}$ with apex $\mathcal{J}$
is the quotient of the path algebra of the following quiver
(here the number $\mathtt{i}$ corresponds to $P_i$):
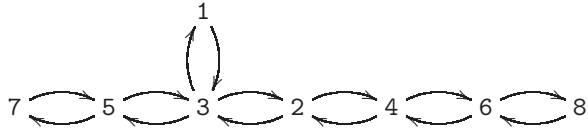
\begin{displaymath}
\xymatrix{
&&\mathtt{1}\ar@/^/[d]&&&&\\
\mathtt{7}\ar@/^/[r]&\mathtt{5}\ar@/^/[r]\ar@/^/[l]&\mathtt{3}\ar@/^/[r]\ar@/^/[l]\ar@/^/[u]&
\mathtt{2}\ar@/^/[r]\ar@/^/[l]&\mathtt{4}\ar@/^/[r]\ar@/^/[l]&\mathtt{6}\ar@/^/[l]\ar@/^/[r]&\mathtt{8}\ar@/^/[l]
}
\end{displaymath}
modulo the relations that any path of the form $\mathtt{i}\to\mathtt{j}\to \mathtt{k}$ is zero
if $i\neq k$ and all paths of the form $\mathtt{i}\to\mathtt{j}\to \mathtt{i}$ coincide.
This algebra is the quadratic dual of the preprojective algebra of the underlying Dynkin quiver
of type $E_8$. For further connections between these two Coxeter groups,
see \cite{De} and references therein.
} 
\end{remark}

\subsection{Weyl type $F_4$}\label{s7.3}

In this subsection we assume that $W$ is of Weyl type $F_4$ and $S=\{r,s,t,u\}$ with the Coxeter diagram
\begin{displaymath}
\xymatrix{r\ar@{-}[r]&s\ar@{-}[r]^4&t\ar@{-}[r]&u.} 
\end{displaymath}

\begin{proposition}\label{prop85}
Assume that $W$ is of Weyl type $F_4$. Let $\mathbf{M}$ be a simple transitive $2$-representation of
$\underline{\cS}$ with apex $\mathcal{J}$. Then the isomorphism classes of indecomposable objects in
$\mathbf{M}(\mathtt{i})$ can be ordered such that, with respect to the ordering
$r,s,t,u$, we have either
\begin{displaymath}
\mathtt{M}=\left(
\begin{array}{cccccc}
2&0&1&0&0&0\\
0&2&0&1&0&0\\
1&0&2&0&1&0\\
0&1&0&2&1&0\\
0&0&1&1&2&1\\
0&0&0&0&1&2
\end{array}
\right)\quad\text{ or }\quad
\mathtt{M}=\left(
\begin{array}{cccccc}
2&1&0&0&0&0\\
1&2&1&1&0&0\\
0&1&2&0&1&0\\
0&1&0&2&0&1\\
0&0&1&0&2&0\\
0&0&0&1&0&2
\end{array}
\right).
\end{displaymath}
\end{proposition}

\begin{proof}
Mutatis mutandis the proof of Proposition~\ref{prop81}.
\end{proof}

\begin{theorem}\label{thm86}
Assume that $W$ is of Weyl type $F_4$. Then every simple transitive $2$-representation of 
$\underline{\cS}$ is equivalent to a cell $2$-representation.
\end{theorem}

\begin{proof}
Mutatis mutandis the proof of Theorem~\ref{thm82}.
\end{proof}

\begin{remark}\label{rem86-1}
{\rm
We have two different cell $2$-representations of $\underline{\cS}$ with apex $\mathcal{J}$ 
in Weyl  type $F_4$. However, the underlying algebras of these  cell $2$-representation are
isomorphic. This common algebra is the quotient of the path algebra of the following 
quiver (here the number $\mathtt{i}$ corresponds to $P_i$ in the left matrix
in Proposition~\ref{prop85}):
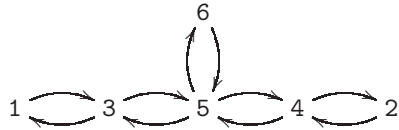
\begin{displaymath}
\xymatrix{
&&\mathtt{6}\ar@/^/[d]&&&&\\
\mathtt{1}\ar@/^/[r]&\mathtt{3}\ar@/^/[r]\ar@/^/[l]&\mathtt{5}\ar@/^/[r]\ar@/^/[l]\ar@/^/[u]&
\mathtt{4}\ar@/^/[r]\ar@/^/[l]&\mathtt{2}\ar@/^/[l]
}
\end{displaymath}
modulo the relations that any path of the form $\mathtt{i}\to\mathtt{j}\to \mathtt{k}$ is zero
if $i\neq k$ and all paths of the form $\mathtt{i}\to\mathtt{j}\to \mathtt{i}$ coincide.
This algebra is the quadratic dual of the preprojective algebra of the underlying Dynkin quiver
of type $E_6$. 
} 
\end{remark}

\subsection{Coxeter type $B_n$}\label{s7.4}

In this subsection we assume that $W$ is of type $B_n$ and $S=\{r,s,t,u,\dots,v\}$ with Coxeter
diagram
\begin{displaymath}
\xymatrix{
r\ar@{-}[r]^{4}&s\ar@{-}[r]&t\ar@{-}[r]&u\ar@{-}[r]&\dots\ar@{-}[r]&v.
}
\end{displaymath}

\begin{proposition}\label{prop87}
Assume that $W$ is of type $B_n$. Let $\mathbf{M}$ be a simple transitive $2$-representation of
$\underline{\cS}$ with apex $\mathcal{J}$. Then the isomorphism classes of indecomposable objects in
$\mathbf{M}(\mathtt{i})$ can be ordered such that, with respect to the ordering
$r,s,t,u,\dots,v$, we have either
\begin{displaymath}
\mathtt{M}=\left(
\begin{array}{cccccccc}
2&0&1&0&0&\dots&0&0\\
0&2&1&0&0&\dots&0&0\\
1&1&2&1&0&\dots&0&0\\
0&0&1&2&1&\dots&0&0\\
0&0&0&1&2&\dots&0&0\\
\vdots&\vdots&\vdots&\vdots&\vdots&\ddots&\vdots&\vdots\\
0&0&0&0&0&\dots&2&1\\
0&0&0&0&0&\dots&1&2
\end{array}
\right)
\end{displaymath}
or
\begin{displaymath}
\mathtt{M}=\left(
\begin{array}{cccccccccccc}
2&1&1&0&0&0&0&\dots&0&0&0&0\\
1&2&0&1&0&0&0&\dots&0&0&0&0\\
1&0&2&0&1&0&0&\dots&0&0&0&0\\
0&1&0&2&0&1&0&\dots&0&0&0&0\\
0&0&1&0&2&0&1&\dots&0&0&0&0\\
0&0&0&1&0&2&0&\dots&0&0&0&0\\
0&0&0&0&1&0&2&\dots&0&0&0&0\\
\vdots&\vdots&\vdots&\vdots&\vdots&\vdots&\vdots&\ddots&\vdots&\vdots&\vdots&\vdots\\
0&0&0&0&0&0&0&\dots&2&0&1&0\\
0&0&0&0&0&0&0&\dots&0&2&0&1\\
0&0&0&0&0&0&0&\dots&1&0&2&0\\
0&0&0&0&0&0&0&\dots&0&1&0&2
\end{array}
\right).
\end{displaymath}
\end{proposition}

\begin{proof}
Mutatis mutandis the proof of  Proposition~\ref{prop81}.

\end{proof}

\begin{theorem}\label{thm88}
Assume that $W$ is of type $B_n$. Then every simple transitive $2$-representation of 
$\underline{\cS}$ is equivalent to a cell $2$-representation.
\end{theorem}

\begin{proof}
Mutatis mutandis the proof of Theorem~\ref{thm82}.
\end{proof}

\begin{remark}\label{rem88-1}
{\rm
We have two different cell $2$-representations of $\underline{\cS}$ with apex $\mathcal{J}$ 
in type $B_n$. For the first one (which corresponds to the first choice in 
Proposition~\ref{prop87}) the underlying algebra is the quotient of the path algebra 
of the following  quiver:
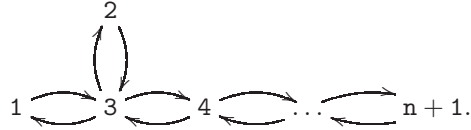
\begin{displaymath}
\xymatrix{
&\mathtt{2}\ar@/^/[d]&&&&\\
\mathtt{1}\ar@/^/[r]&\mathtt{3}\ar@/^/[r]\ar@/^/[l]\ar@/^/[u]&\mathtt{4}\ar@/^/[r]\ar@/^/[l]&
\dots\ar@/^/[r]\ar@/^/[l]&\mathtt{n+1}.\ar@/^/[l]
}
\end{displaymath}
For the second one (which corresponds to the second choice in 
Proposition~\ref{prop87}) the underlying algebra is the quotient of the path algebra 
of the following  quiver:
\begin{displaymath}
\xymatrix{
\mathtt{2n-1}\ar@/^/[r]&
\dots\ar@/^/[r]\ar@/^/[l]&
\mathtt{5}\ar@/^/[r]\ar@/^/[l]&
\mathtt{3}\ar@/^/[r]\ar@/^/[l]&
\mathtt{1}\ar@/^/[r]\ar@/^/[l]&
\mathtt{2}\ar@/^/[r]\ar@/^/[l]&
\mathtt{4}\ar@/^/[r]\ar@/^/[l]&
\dots\ar@/^/[r]\ar@/^/[l]&
\mathtt{2n-2}.\ar@/^/[l]
}
\end{displaymath}
In both cases, we mod out by the relations that any path of the form $\mathtt{i}\to\mathtt{j}\to \mathtt{k}$ is zero
if $i\neq k$ and all paths of the form $\mathtt{i}\to\mathtt{j}\to \mathtt{i}$ coincide.
The first algebra is the quadratic dual of the preprojective algebra of the underlying Dynkin quiver
of type $D_{n+1}$. The second algebra is the quadratic dual of the preprojective algebra of 
the underlying Dynkin quiver of type $A_{2n-1}$.
} 
\end{remark}

\section{New examples of finitary $2$-categories}\label{s8}

\subsection{Symmetric modules}\label{s8.1}

Let $\Bbbk$ be an algebraically closed field. In this subsection we assume that $\mathrm{char}(\Bbbk)\neq 2$. 
Let $A$ be a connected finite dimensional $\Bbbk$-algebra with
a fixed automorphism $\iota:A\to A$ such that $\iota^2=\mathrm{id}_A$. 
For $M\in A\text{-}\mathrm{mod}$, denote by ${}^{\iota}M$ the $A$-module with the same 
underlying space as $M$ but with the new action  $\bullet$ of $A$ twisted by $\iota$:
\begin{displaymath}
a\bullet m:=\iota(a)\cdot m,\quad\text{ for all }\quad a\in A\,\,\text{ and }\,\, m\in M.
\end{displaymath}

For a fixed $A$-module $Q$, consider the category $\mathcal{Q}:=\mathrm{add}(Q\oplus {}^{\iota}Q)$.
Define the category $\hat{\mathcal{Q}}=\hat{\mathcal{Q}}(A,{\iota},Q)$ in the following way:
\begin{itemize}
\item the objects  in $\hat{\mathcal{Q}}$ are all diagrams of the form
\begin{equation}\label{eq92}
\xymatrix{M\ar[rr]^{\alpha}&&{}^{\iota}M} 
\end{equation}
where $M\in \mathcal{Q}$ and $\alpha:M\to {}^{\iota}M$ is an isomorphism in 
$A\text{-}\mathrm{mod}$ such that $\alpha^2\cong \mathrm{id}_M$;
\item morphisms in $\hat{\mathcal{Q}}$ are all commutative diagrams  of the form
\begin{displaymath}
\xymatrix{M\ar[d]_{\varphi}\ar[rr]^{\alpha}&&{}^{\iota}M\ar[d]^{\varphi}\\
N\ar[rr]^{\beta}&&{}^{\iota}N,
} 
\end{displaymath}
where $\varphi:M\to N$ is a homomorphism in $A\text{-}\mathrm{mod}$
(note that $\varphi:{}^{\iota}M\to {}^{\iota}N$ is a homomorphism in $A\text{-}\mathrm{mod}$ as well); 
\item identity morphisms in $\hat{\mathcal{Q}}$ are given by the identity maps;
\item composition in $\hat{\mathcal{Q}}$ is induced from composition in $\mathcal{Q}$
in the obvious way.
\end{itemize}
We will call $\hat{\mathcal{Q}}$ the category of {\em $\iota$-symmetric $A$-modules} over $\mathcal{Q}$. 
Directly from the definitions it follows that $\hat{\mathcal{Q}}$ is additive, $\Bbbk$-linear
and idempotent split.

We note that, if $M$ is an $A$-module such that $M\cong {}^{\iota}M$, then an isomorphism
$\alpha:M\to {}^{\iota}M$ can always be chosen such that $\alpha^2=\mathrm{id}_M$. Indeed,
if $\alpha:M\to {}^{\iota}M$ is any isomorphism, then $\iota^2=\mathrm{id}_A$ implies
that $\alpha^2$ is an automorphism of $M$. As $\alpha^{-2}$ is invertible, there exists
an automorphism $\beta$ of $M$ which is a polynomial in $\alpha^{-2}$ such that $\beta^2=\alpha^{-2}$.
In particular, $\beta$ commutes with $\alpha$.
Then $(\alpha\beta):M\to {}^{\iota}M$ is an isomorphism and $(\alpha\beta)^2=\mathrm{id}_M$.

\begin{proposition}\label{prop91}
The category $\hat{\mathcal{Q}}$ is Krull-Schmidt and has finitely many 
isomorphism classes of indecomposable objects.
\end{proposition}

\begin{proof}
Let $Q_1$, $Q_2$,\dots, $Q_n$ be a complete list of pairwise non-isomorphic indecomposable 
objects in $\mathcal{Q}$. For every $Q_i$, we either have $Q_i\cong {}^{\iota}Q_i$ or
we have $Q_i\cong {}^{\iota}Q_j$, for some $j\neq i$. For each $i$, set 
$Q_i^{(\iota)}:=Q_i\oplus {}^{\iota}Q_i$ and let $\alpha_i:Q_i^{(\iota)}\to {}^{\iota}Q_i^{(\iota)}$
be the homomorphism which swaps the components of the direct sum. Then 
\begin{equation}\label{eqqeenn2}
\xymatrix{Q_i^{(\iota)}\ar[rr]^{\alpha_i}&&{}^{\iota}Q_i^{(\iota)}} 
\end{equation}
is an object in $\hat{\mathcal{Q}}$ which we denote by $(Q_i^{(\iota)},\alpha_i)$.

If $Q_i\cong {}^{\iota}Q_j$, for some $j\neq i$, then $(Q_i^{(\iota)},\alpha_i)$ is, clearly, indecomposable.
Moreover, it is easy to see that $(Q_i^{(\iota)},\alpha_i)$ and $(Q_j^{(\iota)},\alpha_j)$ are isomorphic.

Assume now that $\varphi:Q_i\cong {}^{\iota}Q_i$ is an isomorphism such that 
$\varphi^2=\mathrm{id}_{Q_i}$. Then $\iota^2=\mathrm{id}_A$ implies that
$\varphi:{}^{\iota}Q_i\cong Q_i$ is an isomorphism as well and hence the matrix
\begin{displaymath}
\Phi:=\left(\begin{array}{cc}0&\varphi\\\varphi&0\end{array}\right) 
\end{displaymath}
gives rise to an endomorphism of the object $\alpha_i:Q_i^{(\iota)}\to {}^{\iota}Q_i^{(\iota)}$.
Note that  $\Phi^2$ is the identity on this object.
Write $Q_i^{(\iota)}\cong X_i\oplus Y_i$, where 
$X_i$ denotes the eigenspace of $\Phi$ for the eigenvalue $1$ and 
$Y_i$ denotes the eigenspace of $\Phi$ for the eigenvalue $-1$. Clearly,
$X_i\cong Y_i\cong Q_i$, as $A$-modules.  At the same time, the objects
\begin{displaymath}
\xymatrix{X_i\ar[rr]^{(\alpha_i)|_{X_i}}&&{}^{\iota}X_i} 
\qquad\text{ and }\qquad
\xymatrix{Y_i\ar[rr]^{(\alpha_i)|_{Y_i}}&&{}^{\iota}Y_i} 
\end{displaymath}
are both in $\hat{\mathcal{Q}}$ and are, clearly, indecomposable. 

For each $i$ such that $Q_i\cong {}^{\iota}Q_i$, we fix an isomorphism
$\beta_i:Q_i\to {}^{\iota}Q_i$ such that $(\beta_i)^2=\mathrm{id}_{Q_i}$.
We claim that each indecomposable object in $\hat{\mathcal{Q}}$ is isomorphic to 
$(Q_i^{(\iota)},\alpha_i)$, for some $i$ such that $Q_i\not\cong {}^{\iota}Q_i$,
or is isomorphic to one of the objects
\begin{equation}\label{eqeqnn1}
\xymatrix{Q_i\ar[rr]^{\beta_i}&&{}^{\iota}Q_i}\qquad\text{ or }\qquad
\xymatrix{Q_i\ar[rr]^{-\beta_i}&&{}^{\iota}Q_i}, 
\end{equation}
for some $i$ such that $Q_i\cong {}^{\iota}Q_i$. Indeed, consider an indecomposable 
object of the form \eqref{eq92}. If $M$ contains, as a direct summand, some $N\cong Q_i$ such that
$Q_i\not\cong {}^{\iota}Q_i$, then $N\oplus \alpha(N)$ is a direct summand of $M$
isomorphic to $(Q_i^{(\iota)},\alpha_i)$ and hence $M$ is isomorphic to the latter module.
If $M$ contains, as a direct summand, some $N\cong Q_i$ such that
$Q_i\cong {}^{\iota}Q_i$, then from the previous paragraph it follows that 
$M$ is isomorphic to $\xymatrix{Q_i\ar[rr]^{\alpha}&&{}^{\iota}Q_i}$, for some isomorphism $\alpha$ such that 
$\alpha^2=\mathrm{id}_{Q_i}$. We claim that each such object is isomorphic to one from the list
\eqref{eqeqnn1}. From the commutative diagram
\begin{displaymath}
\xymatrix{ 
Q_i\ar[rrrr]^{\alpha}
\ar[dd]_{\left(\begin{array}{c}\mathrm{id}_{Q_i}\\\alpha\end{array}\right)}
&&&&{}^{\iota}Q_i\ar[dd]^{\left(\begin{array}{c}\mathrm{id}_{{}^{\iota}Q_i}\\\alpha\end{array}\right)}\\\\
Q_i\oplus {}^{\iota}Q_i
\ar[rrrr]^{\left(\begin{array}{cc}0&\mathrm{id}_{{}^{\iota}Q_i}\\\mathrm{id}_{Q_i}&0\end{array}\right)}
\ar[dd]_{\left(\begin{array}{cc}\mathrm{id}_{Q_i}&\alpha\end{array}\right)}
&&&&
{}^{\iota}Q_i\oplus Q_i
\ar[dd]^{\left(\begin{array}{cc}\mathrm{id}_{{}^{\iota}Q_i}&\alpha\end{array}\right)}\\\\
Q_i\ar[rrrr]^{\alpha}&&&&{}^{\iota}Q_i
}
\end{displaymath}
and the fact that $\mathrm{char}(\Bbbk)\neq 2$, it follows that $\xymatrix{Q_i\ar[rr]^{\alpha}&&{}^{\iota}Q_i}$
is a summand of \eqref{eqqeenn2}. As we already established in the previous paragraph, 
\eqref{eqqeenn2} has two direct summands. So, it is enough to argue that the two objects in the 
list \eqref{eqeqnn1} are not isomorphic. The latter follows easily from the fact that $\beta$,
being an automorphism of an indecomposable module, has only one eigenvalue and this eigenvalue is non-zero
and hence $\beta$ cannot be conjugate to $-\beta$ whose unique eigenvalue is different. 

The claim of the proposition follows.
\end{proof}

\subsection{Symmetric projective bimodules}\label{s8.2}

Let $\Bbbk$ be an algebraically closed field and $A$ a finite dimensional $\Bbbk$-algebra with
a fixed automorphism $\iota:A\to A$ such that $\iota^2=\mathrm{id}_A$. We extend 
$\iota$ to an automorphism $\underline{\iota}$ of $A\otimes_{\Bbbk}A^{\mathrm{op}}$ component-wise
and have $\underline{\iota}^2=\mathrm{id}_{A\otimes_{\Bbbk}A^{\mathrm{op}}}$. As usual,
we identify $A\otimes_{\Bbbk}A^{\mathrm{op}}$-modules and $A$-$A$--bimodules. 

Consider the $A$-$A$--bimodule $Q:=A\oplus \left(A\otimes_{\Bbbk}A\right)$ and the corresponding
categories $\mathcal{Q}$ and $\hat{\mathcal{Q}}$. Note that $Q$ is isomorphic to its twist by
$\underline{\iota}$, as the latter amounts to simultaneously twisting both the left and the right actions of
$A$ on $Q$ by $\iota$. The category $\hat{\mathcal{Q}}$ is additive, $\Bbbk$-linear, idempotent split, 
Krull-Schmidt and has finitely many  isomorphism classes of indecomposable objects by 
Proposition~\ref{prop91}. The category of $A$-$A$--bimodules has the natural structure of a tensor
category given by tensor product over $A$. This structure, applied component-wise, turns
$\hat{\mathcal{Q}}$ into a tensor category. We denote by $\cC_{(A,\iota)}$ the strictification 
of $\hat{\mathcal{Q}}$ as described, for example, in \cite[Subsection~2.3]{Le}. Then $\cC_{(A,\iota)}$ 
is a strict  tensor category or, equivalently, a $2$-category with one object which we call $\mathtt{i}$.

\begin{proposition}\label{prop93}
The $2$-category  $\cC_{(A,\iota)}$ is finitary, moreover, it is weakly fiat provided that $A$ is self-injective.
\end{proposition}

\begin{proof}
The fact that  $\cC_{(A,\iota)}$ is finitary follows directly from the construction and
Proposition~\ref{prop91}. If $A$ is self-injective, then the $2$-category $\cC_A$ is weakly
fiat, cf. \cite[Subsection~7.3]{MM1}. The weak anti-automorphism lifts from $\cC_A$ to
$\cC_{(A,\iota)}$ in the obvious way. We claim that even adjunction morphisms can be lifted from 
$\cC_A$ to $\cC_{(A,\iota)}$. Indeed, let $X\in\mathrm{add}(A\oplus \big(A\otimes_{\Bbbk}A\big))$
and $\varphi:X\to A$ be a homomorphism of bimodules. Then the following diagram commutes:
\begin{displaymath}
\xymatrix{
X\oplus {}^{\iota}X^{\iota}\ar[dd]_{\left(\begin{array}{c}
\varphi\\\iota\circ \varphi\end{array}\right)}
\ar[rrr]^{\left(\begin{array}{cc}
0&\mathrm{Id}_X\\\mathrm{Id}_X&0
\end{array}\right)}&&&X\oplus {}^{\iota}X^{\iota}\ar[dd]^{\left(\begin{array}{c}
\varphi\\\iota\circ \varphi\end{array}\right)}\\\\
A\ar[rrr]^{\iota}&&&A
} 
\end{displaymath}
and hence the vertical arrows give $2$-morphisms in $\cC_{(A,\iota)}$. Taking $\varphi$ to be adjunction 
morphisms in $\cC_A$ gives rise, in this way, to adjunction morphisms in $\cC_{(A,\iota)}$. This implies
that $\cC_{(A,\iota)}$ is weakly fiat.
\end{proof}

The novel component of this example compared to various examples which can be found in \cite{MM1}--\cite{MM6},
especially to the $2$-category $\cC_A$ from \cite[Subsection~7.3]{MM1}, is the fact that 
{\em indecomposable} $1$-morphisms in $\cC_{(A,\iota)}$ are given, in general, by {\em decomposable}
endofunctors of $A$-mod. This, in particular, allows us to give an alternative construction 
for \cite[Example~8]{Xa}, see the next subsection.

\subsection{An example}\label{s8.3}

Here we use Subsection~\ref{s8.2} to construct an example of a fiat $2$-category
which has a left cell for which the Duflo involution (see \cite[Proposition~17]{MM1}) is not 
self-adjoint. The example is essentially the same as \cite[Example~8]{Xa}, however, 
it is constructed using completely different methods.

Let  $A$ be the quotient of the path algebra of the quiver
\begin{displaymath}
\xymatrix{\mathtt{1}\ar@/^/[rr]^{\alpha}&&\mathtt{2}\ar@/^/[ll]^{\beta}}
\end{displaymath}
modulo the ideal generated by the relations $\alpha\beta=\beta\alpha=0$. 
Let $\varepsilon_{\mathtt{1}}$ be the trivial path at $\mathtt{1}$ and
$\varepsilon_{\mathtt{2}}$ be the trivial path at $\mathtt{2}$.
Let $\iota$ be the automorphism of $A$ given by swapping 
$\varepsilon_{\mathtt{1}}$ with $\varepsilon_{\mathtt{2}}$
and $\alpha$ with $\beta$. Consider the corresponding $2$-category $\cC_{(A,\iota)}$.

From Subsection~\ref{s8.1} it follows that indecomposable $1$-morphisms in 
$\cC_{(A,\iota)}$ are given by $\mathrm{F}_1:=(A,\iota)$ and $\mathrm{F}_2:=(A,-\iota)$ together with 
\begin{displaymath}
\mathrm{G}_1:=\left(
\big(A\varepsilon_{\mathtt{1}}\otimes_{\Bbbk}\varepsilon_{\mathtt{1}}A\big)\oplus
\big(A\varepsilon_{\mathtt{2}}\otimes_{\Bbbk}\varepsilon_{\mathtt{2}}A\big),
\left(\begin{array}{cc}0&\underline{\iota}\\\underline{\iota}&0\end{array}\right)
\right),
\end{displaymath}
\begin{displaymath}
\mathrm{G}_2:=\left(
\big(A\varepsilon_{\mathtt{1}}\otimes_{\Bbbk}\varepsilon_{\mathtt{2}}A\big)\oplus
\big(A\varepsilon_{\mathtt{2}}\otimes_{\Bbbk}\varepsilon_{\mathtt{1}}A\big),
\left(\begin{array}{cc}0&\underline{\iota}\\\underline{\iota}&0\end{array}\right)
\right).
\end{displaymath}
It is easy to see that we have two two-sided cells, namely,  $\{\mathrm{F}_1,\mathrm{F}_2\}$
and $\{\mathrm{G}_1,\mathrm{G}_2\}$. They both are, at the same time, both left and right cells.
From \cite[Subsection~7.3]{MM1} it follows that the objects $\mathrm{G}_1$ and
$\mathrm{G}_2$ are, in fact, biadjoint to each other. Therefore, $\cC_{(A,\iota)}$ is a fiat $2$-category
and the Duflo involution of the left cell $\{\mathrm{G}_1,\mathrm{G}_2\}$ cannot be self-adjoint
as neither $\mathrm{G}_1$ nor $\mathrm{G}_2$ is.


\noindent
T.~K.: Department of Mathematics, Uppsala University, Box. 480,
SE-75106, Uppsala, SWEDEN. Present address: Department of Mathematics, {\AA}rhus University,
Ny Munkegade 118, 8000 Aarhus C, DENMARK,
email: {\tt kildetoft\symbol{64}math.au.dk}

\noindent
M.~M.: Center for Mathematical Analysis, Geometry, and Dynamical Systems, Departamento de Matem{\'a}tica, 
Instituto Superior T{\'e}cnico, 1049-001 Lisboa, PORTUGAL \& Departamento de Matem{\'a}tica, FCT, 
Universidade do Algarve, Campus de Gambelas, 8005-139 Faro, PORTUGAL, email: {\tt mmackaay\symbol{64}ualg.pt}

\noindent
V.~M.: Department of Mathematics, Uppsala University, Box. 480,
SE-75106, Uppsala, SWEDEN, email: {\tt mazor\symbol{64}math.uu.se}

\noindent
J.~Z.: Department of Mathematics, Uppsala University, Box. 480,
SE-75106, Uppsala, SWEDEN, email: {\tt jakob.zimmermann\symbol{64}math.uu.se}


\begin{thebibliography}{999999}
\bibitem[BL]{BL} A.~Bartels, W.~L{\"u}ck. On crossed product rings with twisted involutions, 
their module categories and L-theory. Cohomology of groups and algebraic K-theory, 1--54, 
\bibitem[BFK]{BFK} J. Bernstein, I. Frenkel, M. Khovanov. A categorification of the Temperley-Lieb algebra
and Schur quotients of $U(\mathfrak{sl}_2)$ via projective and Zuckerman functors. Selecta Math.
(N.S.) \textbf{5} (1999), no. 2, 199--241.
\bibitem[BW]{BW} S.~Billey, G.~Warrington. Kazhdan-Lusztig polynomials for $321$-hexagon-avoiding 
permutations. J. Algebraic Combin. {\bf 13} (2001), no. 2, 111--136. 
\bibitem[Bo]{Bo}  N.~Bourbaki. {\'E}l{\'e}ments de math{\'e}matique. 
Groupes et alg{\`e}bres de Lie. Chapitres {\bf 4}, {\bf 5} et {\bf 6}. Masson, Paris, 1981.
\bibitem[CM]{CM} A.~Chan, V.~Mazorchuk. Diagrams and discrete extensions for finitary $2$-representations.
Preprint arXiv:1601.00080.
\bibitem[CR]{CR} J. Chuang, R. Rouquier. Derived equivalences for symmetric groups and
$\mathfrak{sl}_2$-categorification. Ann. of Math. (2) \textbf{167} (2008), no. 1, 245--298.
\bibitem[CiMa]{CiMa} C.~Cibils, E.~Marcos. Skew category, Galois covering and smash product of 
a $k$-category. Proc. Amer. Math. Soc. {\bf 134} (2006), no. 1, 39--50.
\bibitem[De]{De} P.-P.~Dechant. The $E_8$ geometry from a Clifford perspective. Adv. 
Appl. Clifford Algebr. {\bf 27} (2017), no. 1, 397--421. 
\bibitem[Dh]{Dh} V. Deodhar. A combinatorial settting for questions in Kazhdan-Lusztig theory. 
Geom. Dedicata {\bf 36} (1990), no. 1, 95--119. 
\bibitem[Do]{Do} J.~Douglass. Cells and the reflection representation of Weyl groups and 
Hecke algebras. Trans. Amer. Math. Soc. {\bf 318} (1990), no. 1, 373--399.
\bibitem[Du]{Du} B.~Dubsky. Koszulity of some path categories. Comm. Algebra {\bf 45} (2017), no. 9, 4084--4092.
\bibitem[E]{E} B.~Elias. The two-color Soergel calculus. Compos. Math. {\bf 152} (2016), no. 2, 327--398.
\bibitem[EW]{EW} B.~Elias, G.~Williamson. The Hodge theory of Soergel bimodules. 
Ann. of Math. (2) {\bf 180} (2014), no. 3, 1089--1136. 
\bibitem[ES]{ES} K.~Erdmann, S.~Schroll. Chebyshev polynomials on symmetric matrices. 
Linear Algebra Appl. {\bf 434} (2011), no. 12, 2475--2496.
\bibitem[Fl]{Fl} P.~Flor. On groups of non-negative matrices. Compositio Math. {\bf 21}
(1969), 376--382.
\bibitem[GM1]{GM1} A. L. Grensing, V. Mazorchuk. Categorification of the Catalan monoid.
Semigroup Forum \textbf{89} (2014), no. 1, 155--168.
\bibitem[GM2]{GM2} A. L. Grensing, V. Mazorchuk. Categorification using dual projection functors.
Commun. Contemp. Math. {\bf 19} (2017), no. 3, 1650016, 40 pp. 1
\bibitem[Hu]{Hu} J.~Humphreys. Reflection groups and Coxeter groups. Cambridge Studies in 
Advanced Mathematics, {\bf 29}. Cambridge University Press, Cambridge, 1990.
\bibitem[KaLu]{KaLu} D.~Kazhdan and G.~Lusztig. Representations of Coxeter groups and Hecke 
algebras. Invent. Math. {\bf 53} (1979), 191--213.
\bibitem[Ke]{Ke} B.~Keller. On triangulated orbit categories. Doc. Math. {\bf 10} (2005), 551--581.
\bibitem[KL]{KL} M.~Khovanov, A.~Lauda. A diagrammatic approach to categorification 
of quantum groups. I. Represent. Theory {\bf 13} (2009), 309--347. 
\bibitem[KM1]{KM1} T. Kildetoft, V. Mazorchuk. Parabolic projective functors in type $A$.
Adv. Math. {\bf 301} (2016), 785--803.
\bibitem[KM2]{KM2} T. Kildetoft, V. Mazorchuk. Special modules over positively based algebras.
Doc. Math. {\bf 21} (2016), 1171--1192.
\bibitem[Le]{Le} T. Leinster. Basic bicategories, Preprint arXiv:math/9810017.
\bibitem[Lev]{Lev} D.~Levy. The irreducible factorization of Fibonacci polynomials over $\mathbb{Q}$.
Fibonacci Quart. {\bf 39} (2001), no. 4, 309--319. 
\bibitem[Li]{Li} N.~Libedinsky. {\'E}quivalences entre conjectures de Soergel. 
J. Algebra {\bf 320} (2008), no. 7, 2695--2705. 
\bibitem[Lu1]{Lu1} G.~Lusztig. Cells in affine Weyl groups. Algebraic groups and related topics 
(Kyoto/Nagoya, 1983), 255--287, Adv. Stud. Pure Math. {\bf 6}, North-Holland, Amsterdam, 1985. 
\bibitem[Lu2]{Lu2} G.~Lusztig. Cells in affine Weyl groups. II. J. Algebra {\bf 109} (1987), no. 2, 536--548. 
\bibitem[Lu3]{Lu3} G.~Lusztig. Leading coefficients of character values of Hecke algebras. 
The Arcata Conference on Representations of Finite Groups (Arcata, Calif., 1986), 235--262,
Proc. Sympos. Pure Math. {\bf 47}, Part 2, Amer. Math. Soc., Providence, RI, 1987.
\bibitem[Mac]{Mc} S. Mac Lane. Categories for the Working Mathematician. Springer-Verlag, 1998.
\bibitem[MaMa]{MaMa} M.~Mackaay, V.~Mazorchuk. Simple transitive $2$-representations for 
some $2$-subcategories of Soergel bimodules. J. Pure Appl. Algebra {\bf 221} (2017), no. 3, 565--587.
\bibitem[Maz]{Ma} V. Mazorchuk. Lectures on algebraic categorification. QGM Master Class Series.
European Mathematical Society (EMS), Zurich, 2012
\bibitem[MM1]{MM1} V. Mazorchuk, V. Miemietz. Cell 2-representations of finitary 2-categories.
Compositio Math \textbf{147}  (2011), 1519--1545.
\bibitem[MM2]{MM2} V. Mazorchuk, V. Miemietz. Additive versus abelian 2-representations of fiat 2-categories.
Moscow Math. J. \textbf{14} (2014), No. 3, 595--615.
\bibitem[MM3]{MM3} V. Mazorchuk, V. Miemietz. Endmorphisms of cell 2-representations. 
Int. Math. Res. Not. IMRN \textbf{2016}, no. 24, 7471--7498. 
\bibitem[MM4]{MM4} V. Mazorchuk, V. Miemietz. Morita theory for finitary 2-categories.
Quantum Topol. \textbf{7} (2016), No. 1, 1--28.
\bibitem[MM5]{MM5} V. Mazorchuk, V. Miemietz. Transitive 2-representations of finitary 2-categories.
Trans. Amer. Math. Soc. \textbf{368} (2016), no. 11, 7623--7644.
\bibitem[MM6]{MM6} V. Mazorchuk, V. Miemietz. Isotypic faithful 2-representations of $\mathcal{J}$-simple
fiat 2-categories. Math. Z. \textbf{282} (2016), no. 1-2, 411--434. 
\bibitem[MZ]{MZ} V. Mazorchuk, X.~Zhang. Simple transitive $2$-representations for two non-fiat $2$-categories 
of projective functors. Preprint arXiv:1601.00097. To appear in Ukr. Math. J.
\bibitem[OEIS]{OEIS} N.~Sloane. The online encyclopedia of integer sequences.
https://oeis.org/
\bibitem[Ro]{Ro} R. Rouquier. 2-Kac-Moody algebras. Preprint arXiv:0812.5023.
\bibitem[So1]{So} W. Soergel. The combinatorics of Harish-Chandra bimodules. 
J. Reine Angew. Math. \textbf{429} (1992), 49--74.
\bibitem[So2]{So2} W. Soergel. Kazhdan-Lusztig-Polynome und unzerlegbare Bimoduln {\"u}ber 
Polynomringen. J. Inst. Math. Jussieu \textbf{6} (2007), no. 3, 501--525.
\bibitem[WP]{WP} W.~Webb, E.~ Parberry. Divisibility properties of Fibonacci polynomials. 
Fibonacci Quart. {\bf 7} (1969), no. 5, 457--463.
\bibitem[Xa]{Xa} Q. Xantcha. Gabriel 2-quivers for finitary 2-categories. 
J. London Math. Soc. {\bf 92} (2015), no. 3, 615--632.
\bibitem[Zh1]{Zh1} X. Zhang. Duflo involutions for $2$-categories associated to tree quivers. 
J. Algebra Appl. {\bf 15} (2016), no. 3, 1650041, 25 pp. 
\bibitem[Zh2]{Zh2} X. Zhang. Simple transitive $2$-representations and Drinfeld center for some finitary
$2$-categories. J. Pure Appl. Algebra {\bf 222} (2018), no. 1, 97--130.
\bibitem[Zi]{Zi} J. Zimmermann. Simple transitive $2$-representations of Soergel bimodules in type $B_2$.
J. Pure Appl. Algebra {\bf 221} (2017), no. 3, 666--690. 
\end{thebibliography}
\end{document}